\documentclass[12pt,a4paper,oneside,reqno]{article}

\usepackage[utf8]{inputenc}
\usepackage[english]{babel}
\usepackage[left=1in,right=1in,top=1in,bottom=1in]{geometry}
\usepackage{amsfonts,amssymb,amsmath,amsthm}
\usepackage[symbol]{footmisc}
\usepackage{longtable}
\usepackage{enumerate}
\usepackage{cite}
\usepackage{hyperref}
\hypersetup{colorlinks=true,linkcolor=black,citecolor=black,urlcolor=blue}
\usepackage{fancyhdr}
\theoremstyle{plain}
\newtheorem{theorem}{Theorem}
\newtheorem{lemma}[theorem]{Lemma}

\newtheorem{remark}[theorem]{Remark}

\newtheorem{definition}[theorem]{Definition}
\newtheorem{corollary}[theorem]{Corollary}

\DeclareMathOperator{\Id}{Id}
\DeclareMathOperator{\Aut}{Aut}

\DeclareMathOperator{\Der}{Der}
\DeclareMathOperator{\Ann}{Ann}
\DeclareMathOperator{\GL}{GL}

\usepackage{fouriernc}

\newtheorem*{theoremA}{Theorem A}
\newtheorem*{theoremB}{Theorem B}
\newtheorem*{theoremC}{Theorem C}
\newtheorem*{theoremD}{Theorem D}
\DeclareMathOperator{\Rad}{Rad}
\DeclareMathOperator{\rank}{rank}
\newcommand{\doi}[1]{\href{https://doi.org/#1}{doi:#1}}
\newcommand{\arxiv}[1]{\href{https://arxiv.org/abs/#1}{arXiv:#1}}
\newcommand{\PP}{\mathcal P}
\newcommand{\AAA}{\mathcal A}
\newcommand{\C}{\mathbb C}
\newcommand{\T}{\mathbf T}
\newcommand{\N}{\mathcal N}

\begin{document}

\noindent{\large
The algebraic and geometric classification of  
$3$-Poisson algebras}%
\footnote{The work is supported by FCT 2023.08031.CEECIND, 2023.08952.CEECIND and UID/00212/2025.}%
\footnote{The original concept, the initial list of problems and the overall framework were developed by
H.A., K.A.\ and I.K. The AI system Aristotle (Harmonic, \url{https://aristotle.harmonic.fun/}) assisted
the authors in the literature search, error verification, computer experiments, and the formal
verification in Lean.}

\medskip 

\begin{center}
{\bf
Hani Abdelwahab\footnote{Department of Mathematics, Mansoura University, Mansoura, Egypt; \ haniamar1985@gmail.com},
Kobiljon Abdurasulov\footnote{CMA-UBI, University of Beira Interior, Covilh\~{a}, Portugal; \ Institute of Mathematics Academy of
Sciences of Uzbekistan, Tashkent, Uzbekistan; \ abdurasulov0505@mail.ru} \&
Ivan Kaygorodov\footnote{CMA-UBI, University of Beira Interior, Covilh\~{a}, Portugal; \ kaygorodov.ivan@gmail.com}}
\end{center}

\noindent {\bf Abstract:}
{\it $3$-Poisson (ternary Nambu--Poisson) algebras are the algebras of observables of Nambu mechanics,
the generalization of Hamiltonian mechanics in which the time evolution $\dot f=\{f,H_1,H_2\}$ is
governed by two Hamiltonians and preserves the phase volume; their $3$-Lie part also gives the gauge
algebra of the Bagger--Lambert--Gustavsson model of multiple M2-branes.
We give the algebraic classification of complex $4$-dimensional $3$-Poisson algebras (seven algebras)
and of complex $5$-dimensional nilpotent $3$-Poisson algebras ($38$ algebras and one one-parameter
family) with both multiplications non-zero.
We also describe the algebraic counterpart of the In\"on\"u--Wigner contractions
of physics: the variety of complex $4$-dimensional $3$-Poisson algebras has dimension $16$ and seven
irreducible components, one of which is the orbit closure of the simple $3$-Lie algebra of the
M2-brane model, and the variety of complex $5$-dimensional nilpotent $3$-Poisson algebras has
dimension $20$ and five irreducible components.}

\medskip


\medskip


\tableofcontents

\section*{Introduction}
\addcontentsline{toc}{section}{Introduction}

Algebraic structures with several compatible multiplications appear naturally in mathematics and
mathematical physics \cite{DET26,aae23,MF,kms,CK16,ben,kz,AFM25,ABCF23,LS26,L23,OS24}.
The best known example is given by Poisson algebras, which combine a commutative associative
multiplication and a Lie bracket.
Generalizations of Lie algebras to $n$-ary multiplications were introduced by Filippov~\cite{F85};
structures of this type also appear in the generalization of Hamiltonian mechanics proposed by
Nambu~\cite{N73}.
A \emph{Nambu--Poisson algebra} is a vector space with a commutative associative multiplication and an
$n$-ary skew-symmetric multiplication satisfying the Filippov identity, which are related by the
generalized Leibniz rule; for $n=2$ one recovers Poisson algebras.
The classical example is the polynomial algebra $\C[x_1,x_2,x_3]$ with the Jacobian bracket
\begin{equation*}
[f,g,h]=\det\begin{pmatrix}
\partial_1f&\partial_2f&\partial_3f\\
\partial_1g&\partial_2g&\partial_3g\\
\partial_1h&\partial_2h&\partial_3h
\end{pmatrix}
\end{equation*}
\noindent (see \cite{N73,T94}); the Leibniz rule is the product rule for the derivations $f\mapsto[f,g,h]$.
Filippov classified the $(n+1)$-dimensional $n$-Lie algebras~\cite{F85}; in particular, there is a
unique complex simple $4$-dimensional $3$-Lie algebra.
Leibniz-type identities of higher order also appear in other Poisson-type structures
\cite{DPP96,IDMM97}, and Nambu--Poisson structures on factor algebras of polynomial algebras were
constructed in~\cite{HC23}.

\bigskip
\noindent {\bf Nambu mechanics.}
Nambu's original motivation was to find a dynamics on a phase space of odd dimension that keeps the
main feature of Hamiltonian mechanics, the Liouville theorem on the conservation of the phase
volume~\cite{N73}.
In Nambu mechanics on $\mathbb R^3$ an observable $f$ evolves by $\dot f=\{f,H_1,H_2\}$ with two
Hamiltonians $H_1,H_2$, and the flow preserves the volume form.
The Euler equations of the free rigid body take this form, with the square of the angular momentum
(the Casimir function of $\mathfrak{so}_3$) and the kinetic energy as the two Hamiltonians, so that
both are conserved automatically~\cite{N73,T94}.
Takhtajan observed that the consistency of the time evolution requires the fundamental identity,
which is exactly the Filippov identity, and introduced Nambu--Poisson manifolds as the phase spaces
of this mechanics~\cite{T94}.
The Leibniz rule and the Filippov identity mean that the time evolution is a derivation of both
multiplications of the algebra of observables; thus $3$-Poisson algebras are the algebras of
observables of Nambu mechanics, in the same way as Poisson algebras are the algebras of observables
of Hamiltonian mechanics.
The geometry and dynamics of Nambu--Poisson manifolds have been actively studied since then: their
local structure~\cite{N98}, the dynamics of generalized Poisson and Nambu--Poisson
brackets~\cite{IDMM97,DPP96}, Filippov algebroids and multiplicative structures~\cite{GM00},
reduction by symmetries~\cite{D18}, the Hamilton--Jacobi theory~\cite{DS17}, and the description by
$n$-Lie--Rinehart algebras~\cite{BCZ25}.

\bigskip
\noindent{\bf Quantization.}
The quantization of Nambu mechanics turned out to be much harder than that of Hamiltonian mechanics.
Dito, Flato, Sternheimer and Takhtajan proposed a deformation quantization of Nambu mechanics, the
so-called Zariski quantization, based on the factorization of polynomials~\cite{DFST97}.
Quantized Nambu--Poisson manifolds were related to $n$-Lie algebras by DeBellis, S\"amann and
Szabo~\cite{DSS10} and studied in a $3$-Lie algebra reduced model in~\cite{DSS11}.

\bigskip
\noindent{\bf M2-branes.}
Another source of interest is M-theory.
Basu and Harvey proposed a generalization of the Nahm equation, governed by a ternary bracket, as the
equation describing M2-branes ending on an M5-brane~\cite{BH05}.
Bagger and Lambert~\cite{BL08} and Gustavsson~\cite{G09} constructed a supersymmetric gauge theory of
multiple M2-branes in which the gauge symmetry is given by a $3$-Lie algebra with an invariant
metric.
Papadopoulos~\cite{P08} and Gauntlett and Gutowski~\cite{GG08} proved that the simple
$4$-dimensional $3$-Lie algebra is the only indecomposable non-abelian $3$-Lie algebra with a
positive definite invariant metric, so it is the only possible gauge algebra in this setting.
Poisson-type algebras of currents also appear in two-dimensional integrable field theories, for
example as the chiral Poisson algebras of conformal $\mathcal E$-models~\cite{L23}, and double
Poisson brackets give Poisson structures on representation spaces~\cite{OS24,MF}.

\bigskip
\noindent{\bf $3$-Poisson algebras.}
For $n=3$, non-commutative ternary Nambu--Poisson algebras were constructed in~\cite{MA15}, the simple
linearly compact $n$-Nambu--Poisson algebras were classified in~\cite{CK16}, $3$-Nambu--Poisson
algebras were constructed from $3$-Lie algebras in~\cite{SJ18}, and $\mathcal O$-operators on them
were studied in~\cite{HS25}.
Following~\cite{k23}, we call ternary Nambu--Poisson algebras \emph{$3$-Poisson algebras}. 
In a $3$-Poisson algebra, for all $y,z$ the operator $[-,y,z]$ is a derivation of both
multiplications.
Hence a $3$-Poisson structure on a fixed commutative associative algebra $({\rm A},\cdot)$ is a $3$-Lie
multiplication on ${\rm A}$ all of whose operators $[-,y,z]$ are derivations of $({\rm A},\cdot)$, and
two such structures give isomorphic $3$-Poisson algebras if and only if they are conjugate under the
automorphism group $\Aut({\rm A})$.

\bigskip
\noindent{\bf The algebraic and geometric classification.}
The algebraic and geometric classification of low-dimensional algebras is a classical problem
\cite{ABCF23,AFM25,kz,BC99,FKS,FKS25,GRH,GRH3,KV20,k23,l24,MS}.
The algebraic classification describes the isomorphism classes of algebras of a given dimension;
the geometric classification describes the irreducible components of the variety of structure
constants, which are the closures of the orbits of rigid algebras and of the unions of orbits of
rigid families.
Degenerations generalize the In\"on\"u--Wigner contractions of Lie algebras~\cite{IW53}, which
describe the passage from one physical theory to its limit, as the passage from relativistic to
Galilean mechanics when the speed of light tends to infinity; for example,
the orbit closures of the $4$-dimensional complex Lie algebras were described in~\cite{BC99}, the
varieties of commutative associative algebras were studied in~\cite{Maz80,Poo08}, Filippov algebras
in~\cite{KV20}, and nilpotent Poisson algebras in~\cite{ABCF23,AFM25}.
For nilpotent algebras the main tool is the method of annihilator extensions~\cite{SS78}.
For $3$-Poisson algebras the cocycles are pairs of a symmetric bilinear and a skew-symmetric
trilinear map linked by the Leibniz rule.
On the geometric side, orbits have dimension \mbox{$n^2-\dim\Der$}, and non-degenerations are proved by
closed Borel-stable conditions on the structure constants~\cite{GRH,MS}.
In this language, Theorem~C below shows that the orbit closure of the simple $3$-Lie algebra
$\mathfrak s$ of the BLG model is an irreducible component of the variety of $4$-dimensional
$3$-Poisson algebras.

\bigskip
\noindent{\bf Our results.}
We classify the complex $4$-dimensional $3$-Poisson algebras (Theorem~A: seven algebras, none of them
simple) and the complex $5$-dimensional nilpotent $3$-Poisson algebras (Theorem~B: $38$ algebras and
the family $\PP_{39}^{\alpha}$) with both multiplications non-zero.
The corresponding varieties have dimensions $16$ and $20$ and seven and five irreducible components
(Theorems~C and~D); one of the components in dimension~$4$ is the orbit closure of the
simple $3$-Lie algebra $\mathfrak s$.
All multiplication tables of Theorems~A and~B were checked formally in Lean~4.

\newpage 
\section{The algebraic classification of \texorpdfstring{$3$}{3}-Poisson algebras}\label{sec:alg}

All vector spaces are over the complex field $\C$.

\begin{definition}[see~\cite{T94}]
A \emph{$3$-Poisson algebra} is a vector space ${\rm P}$ equipped with a commutative associative
multiplication $\cdot$ and a $3$-Lie multiplication $[-,-,-]$, that is, a skew-symmetric trilinear
multiplication satisfying the Filippov identity
\begin{equation*}
\big[x_1,x_2,[y_1,y_2,y_3]\big]=\big[[x_1,x_2,y_1],y_2,y_3\big]+\big[y_1,[x_1,x_2,y_2],y_3\big]+\big[y_1,y_2,[x_1,x_2,y_3]\big],
\end{equation*}
such that the Leibniz rule holds:
\begin{equation}\label{id1}
[x_1\cdot x_2,y_2,y_3]=[x_1,y_2,y_3]\cdot x_2+x_1\cdot[x_2,y_2,y_3].
\end{equation}
\end{definition}

In other words, for all $y,z\in{\rm P}$ the operator $[-,y,z]$ is a derivation of both
multiplications.
An \emph{ideal} of ${\rm P}$ is a subspace $I$ with $I\cdot{\rm P}+[I,{\rm P},{\rm P}]\subseteq I$.
The \emph{annihilator} of ${\rm P}$ is the ideal
$\Ann({\rm P})=\big\{x\in{\rm P}:\ x\cdot{\rm P}=[x,{\rm P},{\rm P}]=0\big\}$.
If ${\rm P}=I\oplus\C x$ is a direct sum of ideals with $x\in\Ann({\rm P})$, then $\C x$ is called
an \emph{annihilator component} of ${\rm P}$.
An algebra with non-zero products is \emph{simple} if it has no ideals different from $0$ and ${\rm P}$.

\begin{definition}
A $3$-Poisson algebra ${\rm P}$ is \emph{nilpotent} if ${\rm P}^{(m)}=0$ for some $m\geq 2$, where
${\rm P}^{(1)}={\rm P}$ and, for $n\geq 2$,
\begin{equation*}
{\rm P}^{(n)}=\sum_{i+j=n}{\rm P}^{(i)}\cdot{\rm P}^{(j)}+
\sum_{i+j+k=n+1}[{\rm P}^{(i)},{\rm P}^{(j)},{\rm P}^{(k)}].
\end{equation*}
\end{definition}

We write ${\rm P}^2={\rm P}^{(2)}={\rm P}\cdot{\rm P}+[{\rm P},{\rm P},{\rm P}]$.
If ${\rm P}$ is nilpotent and ${\rm P}\neq 0$, then $\Ann({\rm P})\neq 0$, and ${\rm P}$ is nilpotent
if and only if ${\rm P}/\Ann({\rm P})$ is nilpotent.

In the multiplication tables we write only the non-zero products of basis vectors, up to the
commutativity of the multiplication and the skew-symmetry of the bracket, and we use the short
notation $[ijk]=[e_i,e_j,e_k]$; for example, $[123]=e_4$ means $[e_1,e_2,e_3]=e_4$.
Unless otherwise stated, the parameters of the families of algebras are arbitrary complex numbers.

\subsection{Complex four-dimensional \texorpdfstring{$3$}{3}-Poisson algebras}\label{sec:alg4}

\begin{theoremA}
Let ${\rm P}$ be a complex $4$-dimensional $3$-Poisson algebra with non-zero commutative associative
multiplication and non-zero $3$-Lie multiplication.
Then ${\rm P}$ is isomorphic to one and only one of the following algebras:

\begin{longtable}{l|l|l}
\hline
& commutative associative multiplication & $3$-Lie multiplication\\
\hline
\endfirsthead
\hline
& commutative associative multiplication & $3$-Lie multiplication\\
\hline
\endhead
$\mathcal P^{4}_{1}$ & $e_{1} e_{1} = e_{1}$ & $[234] = e_{2}$ \\
$\mathcal P^{4}_{2}$ & $e_{1} e_{1} = e_{2}$ & $[134] = e_{3}$ \\
$\mathcal P^{4}_{3}$ & $e_{1} e_{1} = e_{2}$ & $[134] = e_{2}$ \\
$\mathcal P^{4}_{4}$ & $e_{1} e_{2} = e_{3}$ & $[124] = e_{4}$ \\
$\mathcal P^{4}_{5}$ & $e_{1} e_{2} = e_{3}$ & $[124] = e_{3}$ \\
$\mathcal P^{4}_{6}$ & $e_{1} e_{2} = e_{4},\ e_{3} e_{3} = e_{4}$ & $[123] = e_{4}$ \\
$\mathcal P^{4}_{7}$ & $e_{1} e_{1} = e_{1},\ e_{1} e_{2} = e_{2},\ e_{1} e_{3} = e_{3},\ e_{1} e_{4} = e_{4}$ & $[234] = e_{2}$ \\
\hline
\end{longtable}
\noindent The algebras $\PP^4_{3}$, $\PP^4_{5}$ and $\PP^4_{6}$ are nilpotent; the other ones are not.
\end{theoremA}

The rest of this subsection contains the proof of Theorem~A.

\subsubsection[The method of Aut(A)-orbits on 3-Lie structures]{The method of $\Aut({\rm A})$-orbits on $3$-Lie structures}

Let $({\rm A},\cdot)$ be a commutative associative algebra.
Denote by ${\rm Z}^{3}({\rm A},{\rm A})$ the set of all trilinear maps
$\theta:{\rm A}\times{\rm A}\times{\rm A}\to{\rm A}$ such that
\begin{enumerate}[(i)]
\item $\theta$ is skew-symmetric;
\item $\theta$ satisfies the Filippov identity (with $[-,-,-]$ replaced by $\theta$);
\item $\theta$ satisfies the Leibniz rule
$\theta(x\cdot y,z,w)=\theta(x,z,w)\cdot y+x\cdot\theta(y,z,w)$ for all $x,y,z,w\in{\rm A}$.
\end{enumerate}
Note that ${\rm Z}^{3}({\rm A},{\rm A})\neq\emptyset$, since $0\in{\rm Z}^{3}({\rm A},{\rm A})$; in general
it is not a vector space, because the Filippov identity is quadratic in $\theta$.
For $\theta\in{\rm Z}^{3}({\rm A},{\rm A})$ we write ${\rm A}_\theta=\big({\rm A},\cdot,[-,-,-]_\theta\big)$,
where $[x,y,z]_\theta=\theta(x,y,z)$.
By definition, ${\rm A}_\theta$ is a $3$-Poisson algebra, and every $3$-Poisson algebra
$\big({\rm A},\cdot,[-,-,-]\big)$ is of this form: it is ${\rm A}_\theta$ for $\theta(x,y,z)=[x,y,z]$.

The automorphism group $\Aut({\rm A})$ of $({\rm A},\cdot)$ acts on ${\rm Z}^{3}({\rm A},{\rm A})$ by
\begin{equation*}
(\theta*\phi)(x,y,z)=\phi^{-1}\big(\theta\big(\phi(x),\phi(y),\phi(z)\big)\big),\qquad
\phi\in\Aut({\rm A}),\ \theta\in{\rm Z}^{3}({\rm A},{\rm A}).
\end{equation*}
It is clear that $\theta*\phi$ again satisfies (i)--(iii).

\begin{lemma}\label{isom}
Let $({\rm A},\cdot)$ be a commutative associative algebra and
$\theta,\vartheta\in{\rm Z}^{3}({\rm A},{\rm A})$.
Then ${\rm A}_\theta$ and ${\rm A}_\vartheta$ are isomorphic if and only if
$\theta*\phi=\vartheta$ for some $\phi\in\Aut({\rm A})$.
\end{lemma}

\begin{proof}
If $\theta*\phi=\vartheta$, then
$\phi\big(\vartheta(x,y,z)\big)=\theta\big(\phi(x),\phi(y),\phi(z)\big)$ for all $x,y,z$.
Since $\phi$ also preserves the commutative associative multiplication,
$\phi:{\rm A}_\vartheta\to{\rm A}_\theta$ is an isomorphism.
Conversely, let $\phi:{\rm A}_\vartheta\to{\rm A}_\theta$ be an isomorphism of $3$-Poisson algebras.
Then $\phi\in\Aut({\rm A})$ and
$\phi\big(\vartheta(x,y,z)\big)=\theta\big(\phi(x),\phi(y),\phi(z)\big)$, that is,
\begin{center}
$\vartheta(x,y,z)=\phi^{-1}\Big(\theta\big(\phi(x),\phi(y),\phi(z)\big)\Big)=(\theta*\phi)(x,y,z)$.
\end{center} 
\end{proof}

Hence the classification of the $3$-Poisson algebras with a given commutative associative part
$({\rm A},\cdot)$ reduces to the description of the $\Aut({\rm A})$-orbits on
${\rm Z}^{3}({\rm A},{\rm A})$, and we proceed as follows:
\begin{enumerate}[(1)]
\item compute ${\rm Z}^{3}({\rm A},{\rm A})$;
\item find the $\Aut({\rm A})$-orbits on ${\rm Z}^{3}({\rm A},{\rm A})$;
\item choose a representative $\theta$ of every orbit and write the algebra ${\rm A}_\theta$.
\end{enumerate}

\subsubsection{Four-dimensional commutative associative algebras}

\begin{theorem}[see~\cite{BG13}]\label{4-dim asscom}
Let ${\rm A}$ be a complex $4$-dimensional commutative associative algebra with non-zero
multiplication.
Then ${\rm A}$ is isomorphic to one of the following algebras:

{\small
\begin{longtable}{l|llllll}
\hline
\endhead
${\rm A}_{01}$ & $e_1e_1=e_1$ & $e_2e_2=e_2$ \\
\hline${\rm A}_{02}$ & $e_1e_1=e_1$ & $e_1e_2=e_2$ \\
\hline${\rm A}_{03}$ & $e_1e_1=e_1$ \\
\hline${\rm A}_{04}$ & $e_1e_1=e_2$ \\
\hline${\rm A}_{05}$ & $e_1e_2=e_3$ \\
\hline${\rm A}_{06}$ & $e_1e_1=e_2$ & $e_1e_2=e_3$ \\
\hline${\rm A}_{07}$ & $e_1e_1=e_1$ & $e_2e_2=e_2$ & $e_3e_3=e_3$ \\
\hline${\rm A}_{08}$ & $e_1e_1=e_1$ & $e_2e_2=e_2$ & $e_2e_3=e_3$ \\
\hline${\rm A}_{09}$ & $e_1e_1=e_1$ & $e_1e_2=e_2$ & $e_1e_3=e_3$ \\
\hline${\rm A}_{10}$ & $e_1e_1=e_1$ & $e_1e_2=e_2$ & $e_1e_3=e_3$ & $e_2e_2=e_3$ \\
\hline${\rm A}_{11}$ & $e_1e_1=e_1$ & $e_2e_2=e_3$ \\
\hline${\rm A}_{12}$ & $e_1e_1=-e_3$ & $e_1e_2=e_4$ & $e_2e_2=e_3$ \\
\hline${\rm A}_{13}$ & $e_1e_2=e_4$ & $e_2e_2=e_3$ \\
\hline${\rm A}_{14}$ & $e_1e_2=e_4$ & $e_3e_3=e_4$ \\
\hline${\rm A}_{15}$ & $e_1e_1=e_2$ & $e_1e_2=e_4$ & $e_3e_3=e_4$ \\
\hline${\rm A}_{16}$ & $e_1e_1=e_2$ & $e_1e_2=e_3$ & $e_1e_3=e_4$ & $e_2e_2=e_4$ \\
\hline${\rm A}_{17}$ & $e_1e_1=e_1$ & $e_2e_2=e_2$ & $e_3e_3=e_3$ & $e_4e_4=e_4$ \\
\hline${\rm A}_{18}$ & $e_1e_1=e_1$ & $e_2e_2=e_2$ & $e_3e_3=e_3$ & $e_3e_4=e_4$ \\
\hline${\rm A}_{19}$ & $e_1e_1=e_1$ & $e_1e_2=e_2$ & $e_3e_3=e_3$ & $e_3e_4=e_4$ \\
\hline${\rm A}_{20}$ & $e_1e_1=e_1$ & $e_2e_2=e_2$ & $e_2e_3=e_3$ & $e_2e_4=e_4$ \\
\hline${\rm A}_{21}$ & $e_1e_1=e_1$ & $e_2e_2=e_2$ & $e_2e_3=e_3$ & $e_2e_4=e_4$ & $e_3e_3=e_4$ \\
\hline${\rm A}_{22}$ & $e_1e_1=e_1$ & $e_1e_2=e_2$ & $e_1e_3=e_3$ & $e_1e_4=e_4$ \\
\hline${\rm A}_{23}$ & $e_1e_1=e_1$ & $e_1e_2=e_2$ & $e_1e_3=e_3$ & $e_1e_4=e_4$ & $e_2e_2=e_3$ \\
\hline${\rm A}_{24}$ & $e_1e_1=e_1$ & $e_1e_2=e_2$ & $e_1e_3=e_3$ & $e_1e_4=e_4$ & $e_2e_2=e_3$ & $e_2e_3=e_4$ \\
\hline${\rm A}_{25}$ & $e_1e_1=e_1$ & $e_1e_2=e_2$ & $e_1e_3=e_3$ & $e_1e_4=e_4$ & $e_2e_2=e_4$ & $e_3e_3=e_4$ \\
\hline${\rm A}_{26}$ & $e_1e_1=e_1$ & $e_2e_2=e_2$ & $e_3e_3=e_4$ \\
\hline${\rm A}_{27}$ & $e_1e_1=e_1$ & $e_1e_2=e_2$ & $e_3e_3=e_4$ \\
\hline${\rm A}_{28}$ & $e_1e_1=e_1$ & $e_2e_3=e_4$ \\
\hline${\rm A}_{29}$ & $e_1e_1=e_1$ & $e_2e_2=e_3$ & $e_2e_3=e_4$ \\
\hline
\end{longtable}
}
\end{theorem}

\noindent 
Note that ${\rm A}_{04},{\rm A}_{05},{\rm A}_{06},{\rm A}_{12},\dots,{\rm A}_{16}$ are exactly
the nilpotent algebras of this list, and the other algebras contain a non-zero idempotent.

\subsubsection[Computation of Z3(A,A)]{Computation of ${\rm Z}^{3}({\rm A},{\rm A})$}

Let ${\rm A}$ be a commutative associative algebra and let $\theta$ be a skew-symmetric trilinear map
on ${\rm A}$ satisfying the Leibniz rule~(iii).
We repeatedly use the following elementary consequences of the Leibniz rule.
\begin{enumerate}[(L1)]
\item Taking $y=x=a$ in (iii): $\theta(a\cdot a,z,w)=2\,a\cdot\theta(a,z,w)$.
In particular, for $z=a$ we get
\begin{center}$\theta(a\cdot a,a,w)=0\quad\text{for all } a,w\in{\rm A}.$\end{center}
\item If $t\cdot{\rm A}=0$, then, taking $x=t$ in (iii), we obtain
$\theta(t,z,w)\cdot y=0$ for all $y$, i.e.
\begin{center}
$\theta(t,z,w)\in\Ann({\rm A}):=\{u:\ u\cdot{\rm A}=0\}$.\end{center}
\item If $e\cdot e=e$, then, by (L1), $\theta(e,z,w)=2e\cdot\theta(e,z,w)$; multiplying by $e$ we get
$e\cdot\theta(e,z,w)=2e\cdot\theta(e,z,w)$, so $e\cdot\theta(e,z,w)=0$ and hence $\theta(e,z,w)=0$.
Then (iii) with $x=e$ gives 
\begin{center}$\theta(e\cdot y,z,w)=e\cdot\theta(y,z,w)$.\end{center}
\end{enumerate}
Recall also that a skew-symmetric trilinear map on a vector space of dimension at most $2$ is zero;
hence, if $\theta(u,-,-)=0$ for all $u$ in a subspace $U$ with $\dim{\rm A}/U\leq2$, then $\theta=0$.

\begin{lemma}\label{lem:idemp}
Let ${\rm A}$ be a $4$-dimensional commutative associative algebra with a non-zero idempotent $e$,
and let $\theta\neq0$ be a skew-symmetric trilinear map satisfying the Leibniz rule.
Then ${\rm A}\cong{\rm A}_{03}$ or ${\rm A}\cong{\rm A}_{22}$, and in the basis of
Theorem~\ref{4-dim asscom} the only non-zero value of $\theta$ on basis vectors (up to
skew-symmetry) is
\[\theta(e_2,e_3,e_4)=\alpha e_2+\beta e_3+\gamma e_4,\qquad(\alpha,\beta,\gamma)\neq(0,0,0).\]
\end{lemma}

\begin{proof}
Put ${\rm A}_1=e\cdot{\rm A}$ and ${\rm A}_0=\{x:\ e\cdot x=0\}$.
Then ${\rm A}={\rm A}_1\oplus{\rm A}_0$ (write $x=e\cdot x+(x-e\cdot x)$),
${\rm A}_1\cdot{\rm A}_0=0$, and $e$ is the identity element of ${\rm A}_1$.
By (L3), 
\begin{center}$\theta(e,-,-)=0$ and $\theta(e\cdot x,z,w)=e\cdot\theta(x,z,w)$.\end{center}
Using skew-symmetry, we also have
\begin{center}
$\theta(x,e\cdot y,z)=-\theta(e\cdot y,x,z)=-e\cdot\theta(y,x,z)=e\cdot\theta(x,y,z)$,\end{center}  and similarly
for the third argument.
Consequently:
\begin{enumerate}
\item if $x\in{\rm A}_0$, then $e\cdot\theta(x,y,z)=\theta(e\cdot x,y,z)=0$, so
$\theta(x,y,z)\in{\rm A}_0$; moreover, for $y\in{\rm A}_1$ we get
$\theta(x,y,z)=\theta(x,e\cdot y,z)=e\cdot\theta(x,y,z)=0$;
\item if $x,y,z\in{\rm A}_1$, then $\theta(x,y,z)=\theta(e\cdot x,y,z)=e\cdot\theta(x,y,z)\in{\rm A}_1$.
\end{enumerate}
Thus $\theta$ vanishes on all "mixed" triples, and it is the sum of its restrictions to ${\rm A}_1$
and to ${\rm A}_0$.
The restriction to ${\rm A}_1$ vanishes when one of the arguments is $e$, so it is zero unless
$\dim{\rm A}_1/\C e\geq3$, i.e.\ ${\rm A}_1={\rm A}$ and $e$ is the identity of ${\rm A}$.
The restriction to ${\rm A}_0$ is zero unless $\dim{\rm A}_0\geq3$, i.e.\ ${\rm A}_1=\C e$.
So we have two cases:
\begin{enumerate}[(a)]
\item ${\rm A}_1=\C e$; we put $N={\rm A}_0$, which is a $3$-dimensional ideal;
\item $e$ is the identity element of ${\rm A}$.
\end{enumerate}
If ${\rm A}$ contains a non-zero idempotent different from the identity element, we apply the above
argument to it and arrive at case (a).
Otherwise $0$ and $1$ are the only idempotents of ${\rm A}$, so ${\rm A}$ is a local algebra and
${\rm A}=\C1\oplus N$, where $N$ is its (nilpotent) maximal ideal.
In both cases ${\rm A}=\C e\oplus N$, where $N$ is a $3$-dimensional ideal, $\theta(e,-,-)=0$, and
therefore
\[\theta(x,y,z)=\omega(x,y,z)\,v\qquad(x,y,z\in N),\]
where $\omega$ is a fixed non-zero skew-symmetric trilinear form on $N$ (a ``volume form'') and
$v\in{\rm A}$ is a non-zero vector.

\begin{enumerate}
\item[{\bf Step 1:}] $N\cdot N=0$.
First, $N$ has no non-zero idempotents: if $f\in N$ were a non-zero idempotent, then
$\theta(f,-,-)=0$ by (L3), which is impossible because $\omega(f,y,z)\neq0$ for suitable $y,z\in N$.
Assume that $a\cdot a\neq0$ for some $a\in N$.
Then $a$ and $a\cdot a$ are linearly independent (if $a\cdot a=\lambda a$ with $\lambda\neq0$, then
$\lambda^{-1}a$ is an idempotent), so we can choose $u\in N$ such that $a,a\cdot a,u$ is a basis of
$N$.
By (L1), $0=\theta(a\cdot a,a,u)=\omega(a\cdot a,a,u)\,v$ with $\omega(a\cdot a,a,u)\neq0$, a
contradiction.
Hence $a\cdot a=0$ for all $a\in N$ and, by polarization,
\begin{center}$x\cdot y=\frac12\big((x+y)\cdot(x+y)-x\cdot x-y\cdot y\big)=0$ for all $x,y\in N$.\end{center}
Therefore ${\rm A}=\C e\oplus N$ with $e\cdot N=0$ and $N\cdot N=0$ (case (a)), i.e.\
${\rm A}\cong{\rm A}_{03}$, or ${\rm A}=\C1\oplus N$ with $N\cdot N=0$ (case (b)), i.e.\
${\rm A}\cong{\rm A}_{22}$.
In both cases $e=e_1$ and $N=\langle e_2,e_3,e_4\rangle$, so the only possibly non-zero value of
$\theta$ is $\theta(e_2,e_3,e_4)=v$.

\item[{\bf Step 2:}] $v\in N$.
In case (a) this was shown above ($\theta$ takes values in ${\rm A}_0=N$).
In case (b) write $v=c e_1+n$ with $n\in N$; then (iii) with $x=y=e_2$, $z=e_3$, $w=e_4$
gives \begin{center}
$0=\theta(e_2\cdot e_2,e_3,e_4)=2e_2\cdot v=2ce_2$, so $c=0$.
\end{center}
\end{enumerate}

Conversely, let ${\rm A}={\rm A}_{03}$ or ${\rm A}_{22}$ and let $\theta(e_2,e_3,e_4)=v\in N$ be the only
non-zero value of $\theta$; we check the Leibniz rule (iii), which is linear in $x$ and in $y$.
If $x,y\in N$, then $x\cdot y=0$ and both terms of the right hand side lie in $N\cdot N=0$.
For ${\rm A}_{03}$ and $x=e_1$: the left hand side $\theta(e_1\cdot y,z,w)$ is zero, because
$e_1\cdot y\in\C e_1$; the right hand side is
\begin{center}$\theta(e_1,z,w)\cdot y+e_1\cdot\theta(y,z,w)=0$,\end{center} because
$\theta(e_1,-,-)=0$ and $e_1\cdot N=0$.
For ${\rm A}_{22}$ and $x=e_1$ (the identity element) both sides equal $\theta(y,z,w)$.
The case $y=e_1$ is symmetric.
\end{proof}

\begin{lemma}\label{lem:nilp}
Let ${\rm A}$ be one of the nilpotent algebras
${\rm A}_{04},{\rm A}_{05},{\rm A}_{06},{\rm A}_{12},\dots,{\rm A}_{16}$, and let $\theta\neq0$ be a
skew-symmetric trilinear map satisfying the Leibniz rule.
Then ${\rm A}\in\{{\rm A}_{04},{\rm A}_{05},{\rm A}_{14}\}$, and the only non-zero value of $\theta$ on
basis vectors (up to skew-symmetry) is given below
\begin{longtable}{l|l}
\hline
${\rm A}$ & the only non-zero value of $\theta$\\
\hline
\endfirsthead
\hline
${\rm A}$ & the only non-zero value of $\theta$\\
\hline
\endhead
${\rm A}_{04}$ & $\theta(e_1,e_3,e_4)=\alpha e_2+\beta e_3+\gamma e_4$\\
\hline
${\rm A}_{05}$ & $\theta(e_1,e_2,e_4)=\alpha e_3+\beta e_4$\\
\hline
${\rm A}_{14}$ & $\theta(e_1,e_2,e_3)=\alpha e_4$\\
\hline
\end{longtable}
\end{lemma}

\begin{proof}
We consider the algebras one by one.
In every case we write $\theta_{ijk}=\theta(e_i,e_j,e_k)$, $i<j<k$; there are four such values.
\begin{enumerate}
\item[{\bf ${\rm A}_{04}$:}] $e_1e_1=e_2$, $\Ann({\rm A})=\langle e_2,e_3,e_4\rangle$.
By (L1) with $a=e_1$, $\theta(e_2,e_1,w)=0$, so $\theta_{123}=\theta_{124}=0$.
By (L1), $\theta_{234}=\theta(e_1e_1,e_3,e_4)=2e_1\cdot\theta_{134}$.
By (L2), $\theta_{134}=\theta(e_3,e_4,e_1)\in\Ann({\rm A})$, so $e_1\cdot\theta_{134}=0$ and
$\theta_{234}=0$.
Hence the only non-zero value is $\theta_{134}\in\langle e_2,e_3,e_4\rangle$.
Conversely, if $\theta_{134}\in\langle e_2,e_3,e_4\rangle$ is the only non-zero value, then both sides of
(iii) vanish: the left hand side because ${\rm A}\cdot{\rm A}=\C e_2$ and $\theta(e_2,-,-)=0$, and the
right hand side because all values of $\theta$ lie in $\Ann({\rm A})$.

\item[{\bf ${\rm A}_{05}$:}] $e_1e_2=e_3$, $\Ann({\rm A})=\langle e_3,e_4\rangle$.
By (L2), every $\theta_{ijk}$ with $k\in\{3,4\}$ lies in $\langle e_3,e_4\rangle$; in particular
$\theta_{124}=\alpha e_3+\beta e_4$.
Taking $x=e_1$, $y=e_2$ in (iii) we get
\[\theta(e_3,z,w)=\theta(e_1,z,w)\cdot e_2+e_1\cdot\theta(e_2,z,w).\]
For $(z,w)=(e_1,e_2)$ the right hand side is $0$, so $\theta_{123}=0$;
for $(z,w)=(e_1,e_4)$ it equals $e_1\cdot\theta(e_2,e_1,e_4)=-e_1\cdot\theta_{124}=0$, so
$\theta_{134}=0$;
for $(z,w)=(e_2,e_4)$ it equals $\theta_{124}\cdot e_2=0$, so $\theta_{234}=0$.
Conversely, if $\theta_{124}=\alpha e_3+\beta e_4$ is the only non-zero value, then both sides of (iii)
vanish: the left hand side because ${\rm A}\cdot{\rm A}=\C e_3$ and $\theta(e_3,-,-)=0$, and the
right hand side because all values of $\theta$ lie in $\Ann({\rm A})$.

\item[{\bf ${\rm A}_{06}$:}] $e_1e_1=e_2$, $e_1e_2=e_3$.
By (L1) with $a=e_1$, $\theta_{123}=\theta_{124}=0$.
Taking $x=e_1$, $y=e_2$, $(z,w)=(e_1,e_4)$ in (iii):
$\theta(e_3,e_1,e_4)=\theta(e_1,e_1,e_4)\cdot e_2+e_1\cdot\theta(e_2,e_1,e_4)=0$, so $\theta_{134}=0$.
Finally, by (L1), $\theta_{234}=2e_1\cdot\theta_{134}=0$. Hence $\theta=0$.

\item[{\bf ${\rm A}_{12}$:}] $e_1e_1=-e_3$, $e_1e_2=e_4$, $e_2e_2=e_3$.
By (L1) with $a=e_1$ and $a=e_2$ we get $\theta(e_3,e_1,w)=\theta(e_3,e_2,w)=0$, i.e.\
$\theta_{123}=\theta_{134}=\theta_{234}=0$.
Taking $x=e_1$, $y=e_2$, $(z,w)=(e_1,e_2)$ in (iii):
$\theta(e_4,e_1,e_2)=\theta(e_1,e_1,e_2)\cdot e_2+e_1\cdot\theta(e_2,e_1,e_2)=0$, so $\theta_{124}=0$.
Hence $\theta=0$.

\item[{\bf ${\rm A}_{13}$:}] $e_1e_2=e_4$, $e_2e_2=e_3$.
By (L1) with $a=e_2$, $\theta_{123}=\theta_{234}=0$.
As for ${\rm A}_{12}$, $\theta(e_4,e_1,e_2)=0$, so $\theta_{124}=0$.
By (L1), $\theta(e_3,e_1,e_4)=2e_2\cdot\theta(e_2,e_1,e_4)=0$, so $\theta_{134}=0$.
Hence $\theta=0$.

\item[{\bf ${\rm A}_{14}$:}] $e_1e_2=e_4$, $e_3e_3=e_4$, $\Ann({\rm A})=\langle e_4\rangle$.
Write $\theta_{123}=c_1e_1+c_2e_2+c_3e_3+c_4e_4$.
Taking $x=e_1,y=e_2$ and $x=y=e_3$ in (iii) we obtain two expressions for $\theta(e_4,z,w)$:
\[\theta(e_4,z,w)=\theta(e_1,z,w)\cdot e_2+e_1\cdot\theta(e_2,z,w)=2e_3\cdot\theta(e_3,z,w).\]
For $(z,w)=(e_1,e_2)$: the first expression is $0$ and the second one is $2e_3\cdot\theta_{123}=2c_3e_4$;
hence $c_3=0$ and $\theta_{124}=0$.
For $(z,w)=(e_1,e_3)$: the first expression is $-e_1\cdot\theta_{123}=-c_2e_4$ and the second one is
$0$; hence $c_2=0$ and $\theta_{134}=0$.
For $(z,w)=(e_2,e_3)$: the first expression is $\theta_{123}\cdot e_2=c_1e_4$ and the second one is $0$;
hence $c_1=0$ and $\theta_{234}=0$.
Thus $\theta_{123}=\alpha e_4$ is the only non-zero value.
Conversely, such a $\theta$ satisfies (iii): the left hand side vanishes because
${\rm A}\cdot{\rm A}=\C e_4$ and $\theta(e_4,-,-)=0$, and the right hand side vanishes because all
values of $\theta$ lie in $\Ann({\rm A})$.

\item[{\bf ${\rm A}_{15}$:}] $e_1e_1=e_2$, $e_1e_2=e_4$, $e_3e_3=e_4$.
By (L1) with $a=e_1$, $\theta_{123}=\theta_{124}=0$.
By (L1) with $a=e_3$, $\theta(e_4,z,w)=2e_3\cdot\theta(e_3,z,w)$, and taking
$(z,w)=(e_1,e_3)$ and $(e_2,e_3)$ we get $\theta_{134}=\theta_{234}=0$.
Hence $\theta=0$.

\item[{\bf ${\rm A}_{16}$:}] $e_1e_1=e_2$, $e_1e_2=e_3$, $e_1e_3=e_4$, $e_2e_2=e_4$.
By (L1) with $a=e_1$, $\theta_{123}=\theta_{124}=0$.
Taking $x=e_1$, $y=e_2$, $(z,w)=(e_1,e_4)$ in (iii) we get $\theta(e_3,e_1,e_4)=0$, so
$\theta_{134}=0$, and then $\theta_{234}=2e_1\cdot\theta_{134}=0$ by (L1).
Hence $\theta=0$.
\end{enumerate}
\end{proof}

\begin{corollary}\label{cor:Z3}
Let ${\rm A}$ be a complex $4$-dimensional commutative associative algebra with non-zero
multiplication.
Then ${\rm Z}^{3}({\rm A},{\rm A})=\{0\}$ unless ${\rm A}$ is isomorphic to one of
${\rm A}_{03}$, ${\rm A}_{04}$, ${\rm A}_{05}$, ${\rm A}_{14}$, ${\rm A}_{22}$, and in these five cases
${\rm Z}^{3}({\rm A},{\rm A})$ consists of the maps described in Lemmas~\ref{lem:idemp}
and~\ref{lem:nilp} (including $\theta=0$).
\end{corollary}

\begin{proof}
By Lemmas~\ref{lem:idemp} and~\ref{lem:nilp}, the maps $\theta$ satisfying (i) and (iii) are exactly
the ones described there.
It remains to check that they satisfy the Filippov identity~(ii).
In all five cases $\theta(x,y,z)=\omega(x,y,z)\,v$, where $v$ is a vector and
$\omega=e_i^*\wedge e_j^*\wedge e_k^*$ for some $i<j<k$ (here $e_1^*,\dots,e_4^*$ is the dual basis).
Fix $a,b\in{\rm A}$ and put $\lambda(u)=\omega(a,b,u)$; then $\theta(a,b,u)=\lambda(u)v$.
The Filippov identity
\begin{center}
$\theta\big(a,b,\theta(y_1,y_2,y_3)\big)=\theta\big(\theta(a,b,y_1),y_2,y_3\big)+\theta\big(y_1,\theta(a,b,y_2),y_3\big)+
\theta\big(y_1,y_2,\theta(a,b,y_3)\big)$ \end{center}
becomes, after comparing the coefficients of $v$,
\[\lambda(v)\,\omega(y_1,y_2,y_3)-\lambda(y_1)\,\omega(v,y_2,y_3)-\lambda(y_2)\,\omega(y_1,v,y_3)-
\lambda(y_3)\,\omega(y_1,y_2,v)=0,\]
and the left hand side is exactly $(\lambda\wedge\omega)(v,y_1,y_2,y_3)$.
Since $\lambda=\omega(a,b,-)$ is a linear combination of $e_i^*,e_j^*,e_k^*$, we have
$\lambda\wedge\omega=0$, and the Filippov identity holds.
\end{proof}

In particular, the underlying commutative associative algebra of a non-trivial
$4$-dimensional $3$-Poisson algebra is isomorphic to one of
${\rm A}_{03},{\rm A}_{04},{\rm A}_{05},{\rm A}_{14},{\rm A}_{22}$.
We now describe the $\Aut({\rm A})$-orbits on the non-zero elements of
${\rm Z}^{3}({\rm A},{\rm A})$ for these algebras.
We write an automorphism $\phi$ by its matrix $(a_{ij})$, where
$\phi(e_j)=\sum_ia_{ij}e_i$ (the $j$-th column is the image of $e_j$).
Note that applying $\phi$ is the same as passing to the new basis $f_j=\phi(e_j)$: the multiplication
table of ${\rm A}_{\theta*\phi}$ in the basis $e_1,\dots,e_4$ is the multiplication table of
${\rm A}_\theta$ in the basis $f_1,\dots,f_4$.

\subsubsection{\texorpdfstring{$3$-Poisson structures on ${\rm A}_{03}$, ${\rm A}_{22}$, ${\rm A}_{04}$, ${\rm A}_{05}$ and ${\rm A}_{14}$}{3-Poisson structures on A03, A22, A04, A05 and A14}}

\noindent{\bf The algebras ${\rm A}_{03}$ and ${\rm A}_{22}$.}
Let ${\rm A}={\rm A}_{03}$ or ${\rm A}={\rm A}_{22}$ and $N=\langle e_2,e_3,e_4\rangle$.
By Lemma~\ref{lem:idemp}, every non-zero $\theta\in{\rm Z}^{3}({\rm A},{\rm A})$ is given by
\[\theta(e_2,e_3,e_4)=v=\alpha e_2+\beta e_3+\gamma e_4\neq0.\]
Since $N\cdot N=0$ and $e_1$ acts on $N$ by $0$ (for ${\rm A}_{03}$) or by the identity (for
${\rm A}_{22}$), every $g\in\GL(N)$ extended by $\phi(e_1)=e_1$ is an automorphism of ${\rm A}$,
and in fact
\[\Aut({\rm A})=\left\{\phi=\begin{pmatrix}1&0&0&0\\0&a_{22}&a_{23}&a_{24}\\0&a_{32}&a_{33}&a_{34}\\
0&a_{42}&a_{43}&a_{44}\end{pmatrix}:\ \det\phi\neq0\right\}.\]
Let $g$ be the lower right $3\times3$ block of $\phi$.
Then $\theta\big(\phi(e_2),\phi(e_3),\phi(e_4)\big)=\det(g)\,v$, so the only non-zero value of $\theta*\phi$ is
\[(\theta*\phi)(e_2,e_3,e_4)=\det(g)\,g^{-1}(v)=\alpha'e_2+\beta'e_3+\gamma'e_4,\qquad
\begin{pmatrix}\alpha'\\\beta'\\\gamma'\end{pmatrix}=\operatorname{adj}(g)
\begin{pmatrix}\alpha\\\beta\\\gamma\end{pmatrix},\]
where $\operatorname{adj}(g)$ is the adjugate matrix; explicitly,
\begin{align*}
\alpha'&=\alpha(a_{33}a_{44}-a_{34}a_{43})+\beta(a_{24}a_{43}-a_{23}a_{44})+\gamma(a_{23}a_{34}-a_{24}a_{33}),\\
\beta'&=\alpha(a_{34}a_{42}-a_{32}a_{44})+\beta(a_{22}a_{44}-a_{24}a_{42})+\gamma(a_{24}a_{32}-a_{22}a_{34}),\\
\gamma'&=\alpha(a_{32}a_{43}-a_{33}a_{42})+\beta(a_{23}a_{42}-a_{22}a_{43})+\gamma(a_{22}a_{33}-a_{23}a_{32}).
\end{align*}
We show that all non-zero $\theta$ lie in one orbit, namely the orbit of $\theta(e_2,e_3,e_4)=e_2$.
\begin{enumerate}
\item If $\alpha\neq0$, take $\phi(e_2)=v$, $\phi(e_3)=e_3$, $\phi(e_4)=\alpha^{-1}e_4$
(so the second column of $\phi$ is $(0,\alpha,\beta,\gamma)^T$).
Then $\det g=1$ and $g^{-1}(v)=e_2$, so $(\theta*\phi)(e_2,e_3,e_4)=e_2$.
\item If $\alpha=0$ and $\beta\neq0$, the automorphism $\phi(e_3)=e_3-e_2$ (other basis vectors fixed)
gives $\alpha'=\beta\neq0$, and we are in the previous case.
\item If $\alpha=\beta=0$, then $\gamma\neq0$, and the automorphism $\phi(e_4)=e_4-e_2$ gives
$\alpha'=\gamma\neq0$, and again we are in the first case.
\end{enumerate}
Therefore, up to isomorphism, there is exactly one non-trivial $3$-Poisson structure on ${\rm A}_{03}$,
namely $\PP^4_1$, and exactly one on ${\rm A}_{22}$, namely $\PP^4_7$.

\medskip\noindent{\bf The algebra ${\rm A}_{04}$.}
By Lemma~\ref{lem:nilp}, every non-zero $\theta\in{\rm Z}^{3}({\rm A}_{04},{\rm A}_{04})$ is given by
\[\theta(e_1,e_3,e_4)=\alpha e_2+\beta e_3+\gamma e_4\neq0.\]
An automorphism must send $e_1$ to an element $\phi(e_1)\notin\Ann({\rm A}_{04})$, $e_2=e_1e_1$ to
$\phi(e_1)\phi(e_1)$, and $\Ann({\rm A}_{04})=\langle e_2,e_3,e_4\rangle$ onto itself, so
\[\Aut({\rm A}_{04})=\left\{\phi=\begin{pmatrix}a_{11}&0&0&0\\a_{21}&a_{11}^2&a_{23}&a_{24}\\
a_{31}&0&a_{33}&a_{34}\\a_{41}&0&a_{43}&a_{44}\end{pmatrix}:\ \det\phi\neq0\right\}.\]
A direct computation shows that the only non-zero value of $\theta*\phi$ is
\begin{center}$(\theta*\phi)(e_1,e_3,e_4)=\alpha'e_2+\beta'e_3+\gamma'e_4$,\end{center} where
\begin{align*}
\alpha'&=a_{11}^{-1}\big(\alpha(a_{33}a_{44}-a_{34}a_{43})+\beta(a_{24}a_{43}-a_{23}a_{44})+
\gamma(a_{23}a_{34}-a_{24}a_{33})\big),\\
\beta'&=a_{11}(\beta a_{44}-\gamma a_{34}),\qquad
\gamma'=a_{11}(\gamma a_{33}-\beta a_{43}).
\end{align*}
In particular, the condition $(\beta,\gamma)=(0,0)$ is preserved by the action.
\begin{enumerate}
\item Let $(\beta,\gamma)\neq(0,0)$.
Choose $w\in\langle e_3,e_4\rangle$ such that the vectors $\beta e_3+\gamma e_4$ and $w$ have
determinant $1$ in the basis $e_3,e_4$ (for example, $w=\beta^{-1}e_4$ if $\beta\neq0$ and
$w=-\gamma^{-1}e_3$ if $\beta=0$), and consider the automorphism
\[\phi(e_1)=e_1,\quad\phi(e_2)=e_2,\quad\phi(e_3)=v=\alpha e_2+\beta e_3+\gamma e_4,\quad\phi(e_4)=w.\]
Then $\theta\big(\phi(e_1),\phi(e_3),\phi(e_4)\big)=v=\phi(e_3)$, i.e.\ $(\theta*\phi)(e_1,e_3,e_4)=e_3$.
We obtain $\PP^4_2$.
(For $(\beta,\gamma)=(1,0)$ this is the automorphism $\phi(e_3)=e_3+\alpha e_2$.)
\item Let $(\beta,\gamma)=(0,0)$; then $\alpha\neq0$, and the automorphism
$\phi(e_4)=\alpha^{-1}e_4$ (other basis vectors fixed) gives $(\theta*\phi)(e_1,e_3,e_4)=e_2$.
We obtain $\PP^4_3$.
\end{enumerate}
The algebras $\PP^4_2$ and $\PP^4_3$ are not isomorphic, since $(\beta,\gamma)=(0,0)$ is an invariant
of the orbit (equivalently, $[{\rm P},{\rm P},{\rm P}]\subseteq{\rm P}\cdot{\rm P}$ holds in $\PP^4_3$ but
not in $\PP^4_2$).
Therefore, up to isomorphism, there are exactly two non-trivial $3$-Poisson structures on
${\rm A}_{04}$.

\medskip\noindent{\bf The algebra ${\rm A}_{05}$.}
By Lemma~\ref{lem:nilp}, every non-zero $\theta\in{\rm Z}^{3}({\rm A}_{05},{\rm A}_{05})$ is given by
\[\theta(e_1,e_2,e_4)=\alpha e_3+\beta e_4\neq0.\]
The automorphisms of ${\rm A}_{05}$ are of two types: either
\begin{center}$\phi(e_1)=a_{11}e_1+a_{31}e_3+a_{41}e_4$, $\phi(e_2)=a_{22}e_2+a_{32}e_3+a_{42}e_4$,\end{center} or
\begin{center}$\phi(e_1)=a_{21}e_2+a_{31}e_3+a_{41}e_4$, $\phi(e_2)=a_{12}e_1+a_{32}e_3+a_{42}e_4$;\end{center}
in both cases $\phi(e_3)=\phi(e_1)\phi(e_2)$ and $\phi(e_4)=a_{34}e_3+a_{44}e_4$, with
$a_{44}\neq0$ and $a_{11}a_{22}\neq0$ (respectively, $a_{12}a_{21}\neq0$).
Put $\Delta=a_{11}a_{22}$ for the first type and $\Delta=-a_{12}a_{21}$ for the second one.
A direct computation shows that the only non-zero value of $\theta*\phi$ is
\[(\theta*\phi)(e_1,e_2,e_4)=\alpha'e_3+\beta'e_4,\qquad
\alpha'=\pm(a_{44}\alpha-a_{34}\beta),\quad\beta'=\Delta\beta,\]
where the sign is $+$ for the first type and $-$ for the second one.
In particular, the condition $\beta=0$ is preserved by the action.
\begin{enumerate}
\item If $\beta\neq0$, take the automorphism
\begin{center}$\phi(e_1)=e_1$, $\phi(e_2)=\beta^{-1}e_2$, $\phi(e_3)=\beta^{-1}e_3$,
$\phi(e_4)=\alpha e_3+\beta e_4$;\end{center} 
then $(\theta*\phi)(e_1,e_2,e_4)=e_4$, and we obtain $\PP^4_4$.
\item If $\beta=0$, then $\alpha\neq0$, and the automorphism $\phi(e_4)=\alpha^{-1}e_4$ (other basis
vectors fixed) gives $(\theta*\phi)(e_1,e_2,e_4)=e_3$; we obtain $\PP^4_5$.
\end{enumerate}
Since $\beta=0$ is an invariant of the orbit, $\PP^4_4\not\cong\PP^4_5$.
Therefore, up to isomorphism, there are exactly two non-trivial $3$-Poisson structures on
${\rm A}_{05}$.

\medskip\noindent{\bf The algebra ${\rm A}_{14}$.}
By Lemma~\ref{lem:nilp}, every non-zero $\theta\in{\rm Z}^{3}({\rm A}_{14},{\rm A}_{14})$ is given by
\[\theta(e_1,e_2,e_3)=\alpha e_4,\qquad\alpha\neq0.\]
The map $\phi(e_1)=e_1$, $\phi(e_2)=\alpha^{-2}e_2$, $\phi(e_3)=\alpha^{-1}e_3$,
$\phi(e_4)=\alpha^{-2}e_4$ is an automorphism of ${\rm A}_{14}$ (indeed,
$\phi(e_1)\phi(e_2)=\alpha^{-2}e_4=\phi(e_3)\phi(e_3)=\phi(e_4)$), and
\[\theta\big(\phi(e_1),\phi(e_2),\phi(e_3)\big)=\alpha^{-3}\alpha e_4=\alpha^{-2}e_4=\phi(e_4),\]
so $(\theta*\phi)(e_1,e_2,e_3)=e_4$.
Hence, up to isomorphism, there is exactly one non-trivial $3$-Poisson structure on ${\rm A}_{14}$,
namely $\PP^4_6$.

\medskip\noindent{\bf End of the proof of Theorem~A.}
By Corollary~\ref{cor:Z3}, the commutative associative part of ${\rm P}$ is isomorphic to one of
${\rm A}_{03},{\rm A}_{04},{\rm A}_{05},{\rm A}_{14},{\rm A}_{22}$, and by Lemma~\ref{isom} and the
orbit computations above, ${\rm P}$ is isomorphic to one of $\PP^4_1,\dots,\PP^4_7$.
Algebras with non-isomorphic commutative associative parts are not isomorphic, and the two pairs
with the same commutative part, $\PP^4_2,\PP^4_3$ and $\PP^4_4,\PP^4_5$, were distinguished above.
Hence the algebras $\PP^4_1,\dots,\PP^4_7$ are pairwise non-isomorphic.

Finally, $\PP^4_3$, $\PP^4_5$ and $\PP^4_6$ are nilpotent, since for suitable degrees of the basis vectors
every non-zero product or bracket of basis vectors has larger degree than each of its arguments ($e_1,e_3,e_4$ of degree $1$ and $e_2$
of degree $2$ in $\PP^4_3$; $e_1,e_2,e_4$ of degree $1$ and $e_3$ of degree $2$ in $\PP^4_5$;
$e_1,e_2,e_3$ of degree $1$ and $e_4$ of degree $2$ in $\PP^4_6$).
The other algebras are not nilpotent: $\PP^4_1$ and $\PP^4_7$ contain the idempotent $e_1$, and in
$\PP^4_2$ and $\PP^4_4$ the operators $[e_1,-,e_4]$ and $[e_1,e_2,-]$ have the eigenvector $e_3$
(respectively, $e_4$) with eigenvalue $1$.
\hfill$\Box$

\begin{remark}[simple algebras]\label{rem:simple4}
None of the algebras $\PP^4_1,\dots,\PP^4_7$ is simple.
Indeed, ${\rm P}^2$ is an ideal of every $3$-Poisson algebra ${\rm P}$, and
${\rm P}^2=\langle e_1,e_2\rangle$, $\langle e_2,e_3\rangle$, $\langle e_3,e_4\rangle$ is a proper non-zero
ideal for ${\rm P}=\PP^4_1$, $\PP^4_2$, $\PP^4_4$, respectively;
in $\PP^4_7$ the subspace $\C e_2$ is an ideal, since $e_2\cdot\PP^4_7=\C e_2$ and
$[e_2,\PP^4_7,\PP^4_7]=\C e_2$;
and the nilpotent algebras $\PP^4_3$, $\PP^4_5$, $\PP^4_6$ have a non-zero annihilator, which is a
proper ideal.
On the other hand, the $3$-Poisson algebras with zero commutative associative multiplication are the
$3$-Lie algebras, and among the $4$-dimensional ones there is exactly one simple algebra, the
Filippov algebra $\mathfrak s$ of 
{\rm Subsection~\ref{sec:geo4} \cite{F85}}; a $4$-dimensional commutative
associative algebra with zero bracket is not simple, since $\C$ is the only simple
finite-dimensional complex commutative associative algebra.
\end{remark}

\subsection{Complex five-dimensional nilpotent \texorpdfstring{$3$}{3}-Poisson algebras}\label{sec:alg5}

\begin{theoremB}
Let ${\rm P}$ be a complex $5$-dimensional nilpotent $3$-Poisson algebra with non-zero commutative
associative multiplication and non-zero $3$-Lie multiplication.
Then ${\rm P}$ is isomorphic to one of the following algebras:

{\small
\begin{longtable}{l|l|l}
\hline
& commutative associative multiplication & $3$-Lie multiplication\\
\hline
\endfirsthead
\hline
& commutative associative multiplication & $3$-Lie multiplication\\
\hline
\endhead
${\mathcal P}_{01}$ & $e_{1} e_{1} = e_{2}$ & $[134] = e_{2}$ \\
\hline${\mathcal P}_{02}$ & $e_{1} e_{2} = e_{3}$ & $[124] = e_{3}$ \\
\hline${\mathcal P}_{03}$ & $e_{1} e_{2} = e_{4},\ e_{3} e_{3} = e_{4}$ & $[123] = e_{4}$ \\
\hline${\mathcal P}_{04}$ & $e_{4} e_{4} = e_{5}$ & $[123] = e_{5}$ \\
\hline${\mathcal P}_{05}$ & $e_{1} e_{1} = e_{5},\ e_{4} e_{4} = e_{5}$ & $[123] = e_{5}$ \\
\hline${\mathcal P}_{06}$ & $e_{1} e_{1} = e_{5},\ e_{2} e_{2} = e_{5},\ e_{4} e_{4} = e_{5}$ & $[123] = e_{5}$ \\
\hline${\mathcal P}_{07}$ & $e_{1} e_{1} = e_{5},\ e_{2} e_{2} = e_{5},\ e_{3} e_{3} = e_{5},\ e_{4} e_{4} = e_{5}$ & $[123] = e_{5}$ \\
\hline${\mathcal P}_{08}$ & $e_{1} e_{4} = e_{5}$ & $[123] = e_{5}$ \\
\hline${\mathcal P}_{09}$ & $e_{1} e_{4} = e_{5},\ e_{2} e_{2} = e_{5}$ & $[123] = e_{5}$ \\
\hline$\mathcal P_{10}$ & $e_{1} e_{4} = e_{5},\ e_{2} e_{2} = e_{5},\ e_{3} e_{3} = e_{5}$ & $[123] = e_{5}$ \\
\hline$\mathcal P_{11}$ & $e_{1} e_{1} = e_{2},\ e_{1} e_{2} = e_{5}$ & $[134] = e_{5}$ \\
\hline$\mathcal P_{12}$ & $e_{1} e_{1} = e_{2},\ e_{1} e_{2} = e_{5},\ e_{3} e_{3} = e_{5}$ & $[134] = e_{5}$ \\
\hline$\mathcal P_{13}$ & $e_{1} e_{1} = e_{2},\ e_{1} e_{2} = e_{5},\ e_{3} e_{3} = e_{5},\ e_{4} e_{4} = e_{5}$ & $[134] = e_{5}$ \\
\hline$\mathcal P_{14}$ & $e_{1} e_{1} = e_{5}$ & $[123] = e_{4},\ [124] = e_{5}$ \\
\hline$\mathcal P_{15}$ & $e_{3} e_{3} = e_{5}$ & $[123] = e_{4},\ [124] = e_{5}$ \\
\hline$\mathcal P_{16}$ & $e_{1} e_{2} = e_{5}$ & $[123] = e_{4},\ [124] = e_{5}$ \\
\hline$\mathcal P_{17}$ & $e_{2} e_{2} = e_{5},\ e_{3} e_{3} = e_{5}$ & $[123] = e_{4},\ [124] = e_{5}$ \\
\hline$\mathcal P_{18}$ & $e_{1} e_{3} = e_{5}$ & $[123] = e_{4},\ [124] = e_{5}$ \\
\hline$\mathcal P_{19}$ & $e_{1} e_{2} = e_{5},\ e_{3} e_{3} = e_{5}$ & $[123] = e_{4},\ [124] = e_{5}$ \\
\hline$\mathcal P_{20}$ & $e_{1} e_{3} = e_{5},\ e_{2} e_{2} = e_{5}$ & $[123] = e_{4},\ [124] = e_{5}$ \\
\hline$\mathcal P_{21}$ & $e_{1} e_{1} = e_{5}$ & $[123] = e_{4}$ \\
\hline$\mathcal P_{22}$ & $e_{1} e_{1} = e_{5},\ e_{2} e_{2} = e_{5}$ & $[123] = e_{4}$ \\
\hline$\mathcal P_{23}$ & $e_{1} e_{1} = e_{5},\ e_{2} e_{2} = e_{5},\ e_{3} e_{3} = e_{5}$ & $[123] = e_{4}$ \\
\hline$\mathcal P_{24}$ & $e_{1} e_{2} = e_{4},\ e_{1} e_{3} = e_{5}$ & $[123] = e_{4}$ \\
\hline$\mathcal P_{25}$ & $e_{1} e_{1} = e_{5},\ e_{2} e_{2} = e_{4}$ & $[123] = e_{4}$ \\
\hline$\mathcal P_{26}$ & $e_{1} e_{1} = e_{4} + e_{5},\ e_{2} e_{2} = e_{5}$ & $[123] = e_{4}$ \\
\hline$\mathcal P_{27}$ & $e_{1} e_{1} = e_{5},\ e_{1} e_{2} = e_{4}$ & $[123] = e_{4}$ \\
\hline$\mathcal P_{28}$ & $e_{1} e_{1} = e_{4},\ e_{1} e_{2} = e_{5}$ & $[123] = e_{4}$ \\
\hline$\mathcal P_{29}$ & $e_{1} e_{1} = e_{4},\ e_{2} e_{2} = e_{4},\ e_{3} e_{3} = e_{5}$ & $[123] = e_{4}$ \\
\hline$\mathcal P_{30}$ & $e_{1} e_{1} = e_{5},\ e_{2} e_{2} = e_{5},\ e_{3} e_{3} = e_{4}$ & $[123] = e_{4}$ \\
\hline$\mathcal P_{31}$ & $e_{1} e_{2} = e_{4},\ e_{2} e_{2} = e_{5},\ e_{3} e_{3} = e_{5}$ & $[123] = e_{4}$ \\
\hline$\mathcal P_{32}$ & $e_{1} e_{2} = e_{5},\ e_{2} e_{2} = e_{4},\ e_{3} e_{3} = e_{4}$ & $[123] = e_{4}$ \\
\hline$\mathcal P_{33}$ & $e_{1} e_{2} = e_{4} + e_{5},\ e_{2} e_{2} = e_{5},\ e_{3} e_{3} = e_{5}$ & $[123] = e_{4}$ \\
\hline$\mathcal P_{34}$ & $e_{1} e_{3} = e_{4},\ e_{2} e_{2} = e_{4},\ e_{2} e_{3} = e_{5}$ & $[123] = e_{4}$ \\
\hline$\mathcal P_{35}$ & $e_{1} e_{3} = e_{5},\ e_{2} e_{2} = e_{5},\ e_{2} e_{3} = e_{4}$ & $[123] = e_{4}$ \\
\hline$\mathcal P_{36}$ & $e_{1} e_{1} = e_{5},\ e_{1} e_{2} = e_{4},\ e_{3} e_{3} = e_{4}$ & $[123] = e_{4}$ \\
\hline$\mathcal P_{37}$ & $e_{1} e_{1} = e_{4},\ e_{1} e_{2} = e_{5},\ e_{3} e_{3} = e_{5}$ & $[123] = e_{4}$ \\
\hline$\mathcal P_{38}$ & $e_{1} e_{1} = e_{5},\ e_{2} e_{2} = e_{4} + e_{5},\ e_{3} e_{3} = e_{4}$ & $[123] = e_{4}$ \\
\hline$\mathcal P_{39}^{\alpha}$ & $e_{1} e_{1} = e_{4} + e_{5},\ e_{2} e_{2} = \alpha e_{4} + e_{5},\ e_{3} e_{3} = e_{5}$ & $[123] = e_{4}$ \\
\hline
\end{longtable}
}

\noindent All algebras listed above are pairwise non-isomorphic, except that
$\PP_{39}^{\alpha}\cong\PP_{39}^{\beta}$ if and only if
$\beta\in\left\{\alpha,\ 1-\alpha,\ \frac1\alpha,\ \frac1{1-\alpha},\ \frac{\alpha}{\alpha-1},\ \frac{\alpha-1}{\alpha}\right\}$
(for $\alpha\in\{0,1\}$ this set is understood as $\{0,1\}$).
\end{theoremB}

\begin{remark}
The complex $5$-dimensional nilpotent $3$-Poisson algebras with zero $3$-Lie multiplication are the
$25$ nilpotent commutative associative algebras $\AAA_1,\dots,\AAA_{25}$
(see~{\rm \cite{Maz80,Poo08}}), and the ones with zero commutative
multiplication are the nilpotent $3$-Lie algebras 
\begin{center}$\C^5;$ \qquad
$\mathfrak n_1\oplus\C:[123]=e_4;$\qquad  and \qquad 
$\mathfrak n_2:[123]=e_4,\ [124]=e_5$ 
(see~{\rm \cite{BSZ11}})$.$ \end{center}
Together with Theorem~B this gives the complete algebraic classification of complex
$5$-dimensional nilpotent $3$-Poisson algebras.
\end{remark}

\subsubsection{The method of annihilator extensions}\label{sec:method}

We use the method of annihilator (central) extensions, which goes back to Skjelbred and
Sund~\cite{SS78}; its adaptation to algebras with two multiplications is described in detail
in~\cite{ABCF23}, and we only recall the main notions.

Let ${\rm P}$ be a $3$-Poisson algebra and $\mathbb V$ a vector space.
Let $\theta:{\rm P}\times{\rm P}\to\mathbb V$ be a symmetric bilinear map and
$\vartheta:{\rm P}\times{\rm P}\times{\rm P}\to\mathbb V$ a skew-symmetric trilinear map.
The pair $(\theta,\vartheta)$ is a \emph{cocycle} if, for all $a,b,c,d,e\in{\rm P}$,
\begin{enumerate}
\item $\theta(a\cdot b,c)=\theta(a,b\cdot c)$;
\item $\vartheta\big(a,b,[c,d,e]\big)=\vartheta\big([a,b,c],d,e\big)+\vartheta\big(c,[a,b,d],e\big)+\vartheta\big(c,d,[a,b,e]\big)$;
\item $\vartheta(a\cdot b,c,d)=\theta\big([a,c,d],b\big)+\theta\big(a,[b,c,d]\big)$.
\end{enumerate}
The cocycles form a vector space ${\rm Z}^{2,3}({\rm P},\mathbb V)$.
On ${\rm P}_{(\theta,\vartheta)}={\rm P}\oplus\mathbb V$ define
\begin{equation*}
(x+u)\cdot(y+v)=x\cdot y+\theta(x,y),\qquad [x+u,y+v,z+w]=[x,y,z]+\vartheta(x,y,z).
\end{equation*}
Then ${\rm P}_{(\theta,\vartheta)}$ is a $3$-Poisson algebra if and only if
$(\theta,\vartheta)\in {\rm Z}^{2,3}({\rm P},\mathbb V)$; it is called an
\emph{$s$-dimensional annihilator extension} of ${\rm P}$, where $s=\dim\mathbb V$.
For a linear map $f:{\rm P}\to\mathbb V$ put $\delta f=\Big(f(x\cdot y),f\big([x,y,z]\big)\Big)$; these
\emph{coboundaries} form a subspace ${\rm B}^{2,3}({\rm P},\mathbb V)$ of ${\rm Z}^{2,3}({\rm P},\mathbb V)$,
and ${\rm H}^{2,3}({\rm P},\mathbb V)={\rm Z}^{2,3}({\rm P},\mathbb V)/{\rm B}^{2,3}({\rm P},\mathbb V)$.
The \emph{radical} of a cocycle is
\begin{center}
$\Rad(\theta,\vartheta)=\big\{x\in{\rm P}:\ \theta(x,{\rm P})=0,\ \vartheta(x,{\rm P},{\rm P})=0\big\}$;\end{center}  then
$\Ann({\rm P}_{(\theta,\vartheta)})=\big(\Rad(\theta,\vartheta)\cap\Ann({\rm P})\big)\oplus\mathbb V$.

Every $n$-dimensional $3$-Poisson algebra $\mathfrak P$ with $\dim\Ann(\mathfrak P)=s\neq0$ is of the
form ${\rm P}_{(\theta,\vartheta)}$, where ${\rm P}=\mathfrak P/\Ann(\mathfrak P)$ and
$(\theta,\vartheta)\in {\rm Z}^{2,3}\big({\rm P},\Ann(\mathfrak P)\big)$ satisfies
$\Rad(\theta,\vartheta)\cap\Ann({\rm P})=0$.
Fix a basis $v_1,\dots,v_s$ of $\mathbb V$ and write
$(\theta,\vartheta)=\sum_{i=1}^s(\theta_i,\vartheta_i)v_i$ with
$(\theta_i,\vartheta_i)\in {\rm Z}^{2,3}({\rm P},\C)$.
The group $\Aut({\rm P})$ acts on ${\rm Z}^{2,3}({\rm P},\C)$ by
$\phi(\theta,\vartheta)=\big(\theta(\phi-,\phi-),\vartheta(\phi-,\phi-,\phi-)\big)$, preserving
${\rm B}^{2,3}({\rm P},\C)$.
Let $\T_s({\rm P})$ be the set of $s$-dimensional subspaces
$\mathbb W=\langle\big[(\theta_1,\vartheta_1)\big],\dots,\big[(\theta_s,\vartheta_s)\big]\rangle$ of
${\rm H}^{2,3}({\rm P},\C)$ such that $\bigcap_{i=1}^s\Rad(\theta_i,\vartheta_i)\cap\Ann({\rm P})=0$.
As in \cite{ABCF23} one proves the following statement.

\begin{theorem}
The $\Aut({\rm P})$-orbits on $\T_s({\rm P})$ are in one-to-one correspondence with the isomorphism
classes of $3$-Poisson algebras without annihilator components which are $s$-dimensional
annihilator extensions of ${\rm P}$ and have $s$-dimensional annihilator.
\end{theorem}

Thus, to classify the $n$-dimensional $3$-Poisson algebras with $s$-dimensional annihilator and
without annihilator components, for every $(n-s)$-dimensional $3$-Poisson algebra ${\rm P}$ we
compute ${\rm H}^{2,3}({\rm P},\C)$ and $\Aut({\rm P})$, describe the $\Aut({\rm P})$-orbits on
$\T_s({\rm P})$, and write the algebra corresponding to a representative of every orbit.

Let $e_1,\dots,e_n$ be a basis of ${\rm P}$.
We denote by $\nabla_{ij}$ ($i\leq j$) the symmetric bilinear form with
$\nabla_{ij}(e_l,e_m)=1$ if $\{l,m\}=\{i,j\}$ and $0$ otherwise, and by $\Delta_{ijk}$ ($i<j<k$) the
skew-symmetric trilinear form with $\Delta_{ijk}(e_i,e_j,e_k)=1$ which vanishes on all other
triples of basis vectors with increasing indices.

\subsubsection{Cohomology of low-dimensional nilpotent \texorpdfstring{$3$}{3}-Poisson algebras}\label{sec:coh}

The nilpotent $3$-Poisson algebras of dimension at most $3$ are commutative associative.
Indeed, the only non-zero skew-symmetric trilinear maps on a $3$-dimensional space ${\rm P}$ are
$[x,y,z]=\omega(x,y,z)v$ with $\omega\neq0$, $v\neq0$, and then $[v,{\rm P},{\rm P}]=\C v$, so
${\rm P}^{(m)}\supseteq\C v$ for all $m$.
The complex nilpotent commutative associative algebras of dimension $3$ are
\begin{center}
$\C^3;$ 
\qquad $e_1e_1=e_2;$ 
\qquad $e_1e_2=e_3;$ 
\qquad and \qquad  $e_1e_1=e_2,\ e_1e_2=e_3;$
\end{center}
and, for a $3$-dimensional nilpotent commutative associative algebra ${\rm P}$, the third cocycle
condition with $[{\rm P},{\rm P},{\rm P}]=0$ gives $\vartheta({\rm P}\cdot{\rm P},{\rm P},{\rm P})=0$.
Hence $(0,\Delta_{123})\in {\rm Z}^{2,3}({\rm P},\C)$ only for ${\rm P}=\C^3$, and
${\rm Z}^{2,3}(\C^3,\C)={\rm H}^{2,3}(\C^3,\C)=\langle\nabla_{ij},\Delta_{123}\rangle$.

The complex $4$-dimensional nilpotent $3$-Poisson algebras are the nilpotent commutative
associative algebras (see~\cite{BG13,Poo08}), the $3$-Lie algebra
$\mathfrak n_1:[123]=e_4$ (see~\cite{BSZ11}), and the algebras $\PP^4_3$, $\PP^4_5$,
$\PP^4_6$ of Theorem~A.
In Table below  we give the cocycles of those of them for which ${\rm Z}^{2,3}({\rm P},\C)$
contains a cocycle with non-zero trilinear part.
For the remaining $4$-dimensional nilpotent commutative associative algebras
\begin{longtable}{lllll}
${\rm A}_{06}:$&$e_1e_1=e_2$&$e_1e_2=e_3$ \\ 
${\rm A}_{12}:$&$e_1e_1=-e_3$&$e_1e_2=e_4$&$e_2e_2=e_3$ \\
${\rm A}_{13}:$&$e_1e_2=e_4$&$e_2e_2=e_3$ \\ 
${\rm A}_{15}:$&$e_1e_1=e_2$&$e_1e_2=e_4$&$e_3e_3=e_4$ \\
${\rm A}_{16}:$&$e_1e_1=e_2$&$e_1e_2=e_3$&$e_1e_3=e_4$&$e_2e_2=e_4$
\end{longtable} 
\noindent every cocycle has zero trilinear part.
For brevity we write $\nabla_{ij}$ and $\Delta_{ijk}$ instead of $(\nabla_{ij},0)$ and
$(0,\Delta_{ijk})$.

\begin{longtable}{l|l|l|l}
\hline
${\rm P}$ & ${\rm Z}^{2,3}({\rm P},\C)$ & ${\rm B}^{2,3}({\rm P},\C)$ & $\Ann({\rm P})$\\
\hline
\endfirsthead
\hline
${\rm P}$ & ${\rm Z}^{2,3}({\rm P},\C)$ & ${\rm B}^{2,3}({\rm P},\C)$ & $\Ann({\rm P})$\\
\hline
\endhead
$\C^4$ & $\nabla_{ij}\ (1\leq i\leq j\leq 4),\ \Delta_{123},\Delta_{124},\Delta_{134},\Delta_{234}$ & $0$ & $\C^4$\\
\hline
$e_1e_1=e_2$ & $\nabla_{11},\nabla_{12},\nabla_{13},\nabla_{14},\nabla_{33},\nabla_{34},\nabla_{44},\Delta_{134}$ & $\nabla_{11}$ & $\langle e_2,e_3,e_4\rangle$\\
\hline
$e_1e_2=e_3$ & $\nabla_{11},\nabla_{12},\nabla_{14},\nabla_{22},\nabla_{24},\nabla_{44},\Delta_{124}$ & $\nabla_{12}$ & $\langle e_3,e_4\rangle$\\
\hline
$e_1e_2=e_4$& $\nabla_{11},\nabla_{12},\nabla_{13},\nabla_{22},\nabla_{23},\nabla_{33},\Delta_{123}$ & $\nabla_{12}+\nabla_{33}$ & $\langle e_4\rangle$\\
$e_3e_3=e_4$ && \\ 
\hline
$\mathfrak n_1$ & $\nabla_{ij}\ (1\leq i\leq j\leq 3),\ \Delta_{123},\Delta_{124},\Delta_{134},\Delta_{234}$ & $\Delta_{123}$ & $\langle e_4\rangle$\\
\hline
$\PP^4_3$ & $\nabla_{11},\nabla_{13},\nabla_{14},\nabla_{33},\nabla_{34},\nabla_{44},\Delta_{134}$ & $\nabla_{11}+\Delta_{134}$ & $\langle e_2\rangle$\\
\hline
$\PP^4_5$ & $\nabla_{11},\nabla_{12},\nabla_{14},\nabla_{22},\nabla_{24},\nabla_{44},\Delta_{124}$ & $\nabla_{12}+\Delta_{124}$ & $\langle e_3\rangle$\\
\hline
$\PP^4_6$ & $\nabla_{11},\nabla_{12},\nabla_{13},\nabla_{22},\nabla_{23},\nabla_{33},\Delta_{123}$ & $\nabla_{12}+\nabla_{33}+\Delta_{123}$ & $\langle e_4\rangle$\\
\hline
\end{longtable}

\subsubsection{Proof of Theorem~B}

\noindent{\bf Reduction to annihilator extensions.}
Let ${\rm P}$ be as in Theorem~B.

\emph{Algebras with an annihilator component.}
If ${\rm P}=I\oplus\C x$ with an annihilator component $\C x$, then $I$ is a $4$-dimensional
nilpotent $3$-Poisson algebra with both multiplications non-zero, i.e.\
$I\in\{\PP^4_3,\PP^4_5,\PP^4_6\}$, and we obtain $\PP_{01}$, $\PP_{02}$, $\PP_{03}$.

From now on ${\rm P}$ has no annihilator components.
Let $s=\dim\Ann({\rm P})$ and ${\rm P}'={\rm P}/\Ann({\rm P})$.
The $3$-Lie multiplication of ${\rm P}$ induces a skew-symmetric trilinear map on the
$(5-s)$-dimensional space ${\rm P}'$ with values in ${\rm P}$; hence $s\leq2$.

\medskip 
\noindent \emph{The case $s=2$.}
Then ${\rm P}'$ is a $3$-dimensional nilpotent commutative associative algebra and the trilinear part
of the cocycle is non-zero, so ${\rm P}'=\C^3$ by Subsection~\ref{sec:coh};
see Case~4 below.

\medskip 
\noindent \emph{The case $s=1$.}
Then ${\rm P}'$ is a $4$-dimensional nilpotent $3$-Poisson algebra.
If ${\rm P}'$ is commutative associative, the cocycle must have a non-zero trilinear part;
if ${\rm P}'\in\{\PP^4_3,\PP^4_5,\PP^4_6\}$, or if ${\rm P}'$ is one of the algebras $e_1e_2=e_3$ and
$e_1e_2=e_4,\ e_3e_3=e_4$, then by Table of cocycles $\Rad(\theta,\vartheta)\cap\Ann({\rm P}')\neq0$ for every cocycle (for $e_1e_2=e_3$ the element
$e_3\in\Ann({\rm P}')$ is not involved in any cocycle, and in the remaining cases no cocycle involves
$\Ann({\rm P}')$ at all), which is impossible.
Therefore ${\rm P}'\in\{\C^4,\ e_1e_1=e_2,\ \mathfrak n_1\}$; these cases are treated in
Cases~1, 2 and~3 below.
 
\begin{enumerate}
    \item[{\bf Case 1:}] one-dimensional annihilator extensions of $\C^4$.

Here ${\rm H}^{2,3}(\C^4,\C)={\rm Z}^{2,3}(\C^4,\C)$ and $\Aut(\C^4)=\GL_4(\C)$.
Let $\mathbb W=\langle\big[(\theta,\vartheta)\big]\rangle\in\T_1(\C^4)$ with $\vartheta\neq0$.
Every non-zero skew-symmetric trilinear form on $\C^4$ is decomposable, so we may assume
$\vartheta=\Delta_{123}$; the radical of $\Delta_{123}$ is $\langle e_4\rangle$, and the stabilizer
of the line $\C\Delta_{123}$ consists of the automorphisms with $\phi(e_4)\in\C e_4$.
The condition $\Rad(\theta,\vartheta)=0$ means $\theta(e_4,-)\neq0$.

If $\theta(e_4,e_4)\neq0$, then replacing $e_i$ by $e_i-\frac{\theta(e_i,e_4)}{\theta(e_4,e_4)}e_4$
($i\leq3$) we get $\theta=c\nabla_{44}+\eta$, where $\eta$ is a symmetric form on
$\langle e_1,e_2,e_3\rangle$.
Using $\GL_3(\C)$ on $\langle e_1,e_2,e_3\rangle$ (which multiplies $\Delta_{123}$ by the
determinant) we reduce $\eta$ to $0$, $\nabla_{11}$, $\nabla_{11}+\nabla_{22}$ or
$\nabla_{11}+\nabla_{22}+\nabla_{33}$, and a diagonal automorphism normalizes the coefficients.
We obtain $\PP_{04}$, $\PP_{05}$, $\PP_{06}$, $\PP_{07}$ (with $e_5$ spanning $\mathbb V$).

If $\theta(e_4,e_4)=0$, then we may assume $\theta(e_4,e_1)=1$ and $\theta(e_4,e_2)=\theta(e_4,e_3)=0$;
replacing $e_1$ by $e_1-\frac12\theta(e_1,e_1)e_4$ and $e_j$ by $e_j-\theta(e_1,e_j)e_4$ ($j=2,3$),
which does not change $\Delta_{123}$, we get $\theta=\nabla_{14}+\eta$, where $\eta$ is a symmetric
form on $\langle e_2,e_3\rangle$.
Reducing $\eta$ to $0$, $\nabla_{22}$ or $\nabla_{22}+\nabla_{33}$ and normalizing the coefficients,
we obtain $\PP_{08}$, $\PP_{09}$, $\PP_{10}$.

These algebras are pairwise non-isomorphic: in all of them $\Ann({\rm P})=\langle e_5\rangle$, the
subspace $K=\big\{x:\ [x,{\rm P},{\rm P}]=0\big\}$ equals $\langle e_4,e_5\rangle$, and the algebras are
distinguished by the rank of the form induced by the multiplication on ${\rm P}/\Ann({\rm P})$ and by
the condition $K\cdot K=0$.

\item[{\bf Case 2:}] one-dimensional annihilator extensions of $e_1e_1=e_2$.

Let ${\rm P}'$ be the algebra $e_1e_1=e_2$ (with basis $e_1,\dots,e_4$).
By Table of cocycles, every $\mathbb W\in\T_1({\rm P}')$ is spanned by
$\big[(\theta,\vartheta)\big]$ with
$\theta=\alpha_{12}\nabla_{12}+\alpha_{13}\nabla_{13}+\alpha_{14}\nabla_{14}+\eta$,
where $\eta$ is a symmetric form on $\langle e_3,e_4\rangle$, and
$\vartheta=\alpha_{134}\Delta_{134}$.
Since $e_2\notin\Rad(\theta,\vartheta)$ we have $\alpha_{12}\neq0$, and since the bracket is non-zero,
$\alpha_{134}\neq0$.
The automorphism $e_3\mapsto e_3-\frac{\alpha_{13}}{\alpha_{12}}e_2$,
$e_4\mapsto e_4-\frac{\alpha_{14}}{\alpha_{12}}e_2$ gives $\alpha_{13}=\alpha_{14}=0$.
The automorphisms acting by $g\in\GL_2(\C)$ on $\langle e_3,e_4\rangle$ reduce $\eta$ to $0$,
$\nabla_{33}$ or $\nabla_{33}+\nabla_{44}$, and an automorphism
$\mathrm{diag}(a,a^2,b,b)$ normalizes the coefficients.
We obtain $\PP_{11}$, $\PP_{12}$, $\PP_{13}$, which are distinguished by the rank of $\eta$
(the rank of the multiplication restricted to $K=\big\{x:\ x\cdot{\rm P}\subseteq\Ann({\rm P})\big\}$
modulo ${\rm P}^2$).

\item[{\bf Case 3:}] one-dimensional annihilator extensions of $\mathfrak n_1$.

Let $\mathbb W=\langle\big[(\theta,\vartheta)\big]\rangle\in\T_1(\mathfrak n_1)$.
By Table of cocycles, $\theta$ is a symmetric form on $L=\langle e_1,e_2,e_3\rangle$, and modulo
${\rm B}^{2,3}$ we can write $\vartheta(x,y,e_4)=\omega(x,y)$ for a skew-symmetric form $\omega$ on $L$,
$\vartheta(L,L,L)=0$.
Since $e_4\notin\Rad(\theta,\vartheta)$, $\omega\neq0$, so $\omega$ has rank $2$ and a kernel line
$\ell\subseteq L$.
Every $g\in\GL(L)$ extends to an automorphism of $\mathfrak n_1$ with $e_4\mapsto(\det g)e_4$, and it
acts by $\theta\mapsto g^t\theta g$, $\omega\mapsto(\det g)g^t\omega g$.
Since the skew-symmetric forms with a given kernel form a line, we may assume
$\omega=\Delta_{124}$, i.e.\ $\ell=\langle e_3\rangle$, and it remains to classify the non-zero
symmetric forms $\theta$ on $L$ up to the stabilizer of $\ell$ and scalars
($\theta=0$ gives $\mathfrak n_2$).
The orbits are determined by $\rank\theta$, by the position of $\ell$ with respect to
$\Rad(\theta)$, and by whether $\ell$ is isotropic; see Table below.
\begin{longtable}{c|l|l|c}
\hline
$\rank\theta$ & position of $\ell$ & $\theta$ & algebra\\
\hline
\endfirsthead
\hline
$\rank\theta$ & position of $\ell$ & $\theta$ & algebra\\
\hline
\endhead
$1$ & $\ell\subseteq\Rad\theta$ & $\nabla_{11}$ & $\PP_{14}$\\
\hline
$1$ & $\theta(\ell,\ell)\neq0$ & $\nabla_{33}$ & $\PP_{15}$\\
\hline
$2$ & $\ell=\Rad\theta$ & $\nabla_{12}$ & $\PP_{16}$\\
\hline
$2$ & $\theta(\ell,\ell)\neq0$ & $\nabla_{22}+\nabla_{33}$ & $\PP_{17}$\\
\hline
$2$ & $\ell\neq\Rad\theta$, $\theta(\ell,\ell)=0$ & $\nabla_{13}$ & $\PP_{18}$\\
\hline
$3$ & $\theta(\ell,\ell)\neq0$ & $\nabla_{12}+\nabla_{33}$ & $\PP_{19}$\\
\hline
$3$ & $\theta(\ell,\ell)=0$ & $\nabla_{13}+\nabla_{22}$ & $\PP_{20}$\\
\hline
\end{longtable}
For example, if $\rank\theta=3$ and $\theta(\ell,\ell)=0$, choose $e_3\in\ell$, then 
\begin{center}$e_1$ with
$\theta(e_1,e_3)=1$, $\theta(e_1,e_1)=0$, and $e_2\in\langle e_1,e_3\rangle^{\perp}$ with
$\theta(e_2,e_2)=1$;\end{center} the other cases are similar.
Choosing the kernel line $\ell=\langle e_1\rangle$ or $\ell=\langle e_2\rangle$ instead of
$\langle e_3\rangle$, i.e.\ $\omega=\Delta_{234}$ or $\omega=\Delta_{134}$, one obtains for
$\theta=\nabla_{12}$ the algebras
\begin{center}
$e_1e_2=e_5,\ [123]=e_4,\ [234]=e_5$ \quad and \quad
$e_1e_2=e_5,\ [123]=e_4,\ [134]=e_5$.
\end{center}
We discard them, since they are isomorphic to the already found algebra $\PP_{18}$ (via
$f_1=e_2,\,f_2=e_3,\,f_3=e_1,\,f_4=e_4,\,f_5=e_5$ and
$f_1=-e_1,\,f_2=e_3,\,f_3=e_2,\,f_4=e_4,\,f_5=-e_5$, respectively), and this is enough.
Finally, the automorphisms $\mathrm{diag}(a,b,c,abc,a^2b^2c)$ normalize all coefficients to $1$.
The seven algebras are pairwise non-isomorphic, because $\ell$ is determined by ${\rm P}$ as the
image in ${\rm P}/{\rm P}^2$ of $\big\{x:\ [x,{\rm P},{\rm P}^2]=0\big\}$.

\item[{\bf Case 4:}] two-dimensional annihilator extensions of $\C^3$.

Let $\mathbb W\in\T_2(\C^3)$ define an algebra with non-zero $3$-Lie multiplication.
Then $\mathbb W$ contains an element with non-zero trilinear part, and, since $\Rad(\Delta_{123})=0$,
the condition $\Rad\cap\Ann=0$ holds automatically.
Hence $\mathbb W=\langle(Q_1,\Delta_{123}),(Q_2,0)\rangle$, where $Q_1,Q_2$ are symmetric
forms on $\C^3$, $Q_2\neq0$, and $Q_1$ is defined modulo $\C Q_2$.
The corresponding algebra is
\begin{equation*}
x\cdot y=Q_1(x,y)e_4+Q_2(x,y)e_5,\qquad [123]=e_4.
\end{equation*}
An element $g\in\GL_3(\C)=\Aut(\C^3)$ maps $\mathbb W$ to
$\langle\big(g^tQ_1 g,\,\det(g)\Delta_{123}\big),(g^tQ_2 g,0)\rangle$.
Using scalar matrices, we see that $Q_1$ and $Q_2$ can be rescaled independently.
Therefore the orbits of $\Aut(\C^3)$ on such subspaces are in one-to-one correspondence with the
$\GL_3(\C)$-orbits of the following data:
\begin{enumerate}[(a)]
\item a non-zero symmetric form $Q_2$ up to scalars, if $Q_1\in\C Q_2$;
\item a pencil $\langle Q_1,Q_2\rangle$ of ternary quadratic forms of dimension $2$ together
with a marked member $[Q_2]$, otherwise.
\end{enumerate}

In case (a), $Q_2$ is determined by its rank, and we obtain $\PP_{21}$, $\PP_{22}$, $\PP_{23}$.

In case (b) we use the classification of pencils of complex symmetric matrices
(see~\cite{Th91}).
A pencil is \emph{singular} if all its members are degenerate, and \emph{regular} otherwise; in the
regular case $\det(sQ_1+tQ_2)$ is a non-zero binary cubic form, its roots are the degenerate
members, and the pencil is determined by its Segre symbol.
The orbits of $2$-dimensional pencils of ternary quadratic forms are listed in Table below
(in coordinates $x_1,x_2,x_3$).
\begin{longtable}{ll}
\hline
type & pencil\\
\hline
\endfirsthead
\hline
type & pencil\\
\hline
\endhead
singular & $\langle x_1^2,x_2^2\rangle$\\
\hline
singular & $\langle x_1^2,x_1x_2\rangle$\\
\hline
singular & $\langle x_1x_2,x_1x_3\rangle$\\
\hline
$[3]$ & $\langle x_1x_3+x_2^2,\,x_2x_3\rangle$\\
\hline
$\big[(21)\big]$ & $\langle x_1^2,\,x_1x_2+x_3^2\rangle$\\
\hline
$[21]$ & $\langle x_1x_2,\,x_2^2+x_3^2\rangle$\\
\hline
$\big[(11)1\big]$ & $\langle x_1^2+x_2^2,\,x_3^2\rangle$\\
\hline
$[111]$ & a pencil in $\langle x_1^2,x_2^2,x_3^2\rangle$ with three distinct degenerate members\\
\hline
\end{longtable}
The stabilizer of a pencil acts on the projective line of its members preserving the degenerate
members and their ranks.
In all cases except $[111]$ this action is transitive on every set of members with the same rank
and the same multiplicity as a root of the determinant (it contains the scalings
$(s:t)\mapsto(as:bt)$ in suitable coordinates, or, for $[3]$, the translations induced by
$x_1\mapsto x_1+\lambda x_2,\ x_2\mapsto x_2+\frac{\lambda}{2}x_3$).
This gives the algebras of Table below.
\begin{longtable}{ll}
\hline
pencil and marked member & algebras\\
\hline
\endfirsthead
\hline
pencil and marked member & algebras\\
\hline
\endhead
$\langle x_1^2,x_2^2\rangle$, marked member of rank $1$ or $2$: & $\PP_{25}$, $\PP_{26}$\\
\hline
$\langle x_1^2,x_1x_2\rangle$, marked member of rank $1$ or $2$: & $\PP_{27}$, $\PP_{28}$\\
\hline
$\langle x_1x_2,x_1x_3\rangle$: & $\PP_{24}$\\
\hline
$\big[(11)1\big]$, marked member of rank $1$,& \\
 of rank $2$ and degenerate, non-degenerate:  & $\PP_{29}$, $\PP_{30}$, $\PP_{39}^{0}$\\
\hline
$[21]$, marked double root, simple root, & \\
 non-degenerate member: &  $\PP_{31}$, $\PP_{32}$, $\PP_{33}$\\ 
 \hline
$[3]$, marked degenerate, non-degenerate member: & $\PP_{34}$, $\PP_{35}$\\
\hline
$\big[(21)\big]$, marked degenerate, non-degenerate member: & $\PP_{36}$, $\PP_{37}$ \\
\hline
\end{longtable}
For the pencil $\langle x_1x_2,x_1x_3\rangle$ the other choice of the marked member gives the algebra
$e_1e_2=e_4,\ e_1e_3=e_5,\ [123]=e_5$; we discard it, since it is isomorphic to the already found
algebra $\PP_{24}$ (via $f_1=e_1,\,f_2=e_3,\,f_3=-e_2,\,f_4=e_5,\,f_5=-e_4$), and this is enough.
In the case $[111]$ the three degenerate members are distinct points of the projective line, the
image of the stabilizer of the pencil is the group $S_3$ of all permutations of these three points,
and the three degenerate members form one orbit.
A degenerate marked member gives $\PP_{38}$.
A non-degenerate marked member is determined, up to the action of $S_3$, by its cross-ratio with
the three degenerate members; for the pencil of $\PP_{39}^{\alpha}$ the degenerate members are
$Q_1$, $Q_1-Q_2$ and $Q_1-\alpha Q_2$, and we obtain the algebras $\PP_{39}^{\alpha}$,
$\alpha\neq0,1$.
Two such algebras are isomorphic if and only if the parameters lie in the same orbit of the
anharmonic group generated by $\alpha\mapsto1-\alpha$ and $\alpha\mapsto\frac1\alpha$; explicitly,
\begin{gather*}
\PP_{39}^{\alpha}\cong\PP_{39}^{1/\alpha}:\ \ f_1=\alpha e_2,\ f_2=\alpha e_1,\ f_3=-\alpha e_3,\ f_4=\alpha^3e_4,\
f_5=\alpha^2e_5;\\
\PP_{39}^{\alpha}\cong\PP_{39}^{1-\alpha}:\ \ f_1=e_3,\ f_2=e_2,\ f_3=e_1,\ f_4=-e_4,\ f_5=e_4+e_5.
\end{gather*}
The algebras with a non-degenerate marked member of a pencil of type $\big[(11)1\big]$ are
$\PP_{39}^{0}\cong\PP_{39}^{1}$.
The isomorphism class of the algebra determines the pencil (as the image of
${\rm P}\cdot{\rm P}$ in the space of forms on ${\rm P}/\Ann({\rm P})$), the marked member (the forms
vanishing on $[{\rm P},{\rm P},{\rm P}]$), and the $j$-invariant
$j(\alpha)=\frac{(\alpha^2-\alpha+1)^3}{\alpha^2(\alpha-1)^2}$ of the cross-ratio; this proves that the algebras
are pairwise non-isomorphic, which completes the proof of Theorem~B.
\end{enumerate}\hfill$\Box$

\section{The geometric classification}\label{sec:geo}

Let $\mathbb V$ be a complex vector space of dimension $n$ with a fixed basis $e_1,\dots,e_n$.
A pair of a symmetric bilinear map $\mu_1:\mathbb V\times\mathbb V\to\mathbb V$ and a
skew-symmetric trilinear map $\mu_2:\mathbb V^{3}\to\mathbb V$ is determined by its structure
constants $c_{ij}^k$, $d_{ijl}^k\in\C$.
The set $\mathfrak P_n$ of pairs $\mu=(\mu_1,\mu_2)$ that define $3$-Poisson algebras is an affine
variety, given by the polynomial equations which express the associativity, the Filippov identity
and the Leibniz rule, and the set $\N_n\subseteq\mathfrak P_n$ of nilpotent $3$-Poisson algebras is a
closed subvariety (nilpotency is a closed condition, since ${\rm P}^{(2^n)}=0$ for every nilpotent
${\rm P}$ of dimension $n$).
The group $\GL(\mathbb V)$ acts on $\mathfrak P_n$ and on $\N_n$ by
\begin{center}$(g*\mu)(x,y)=g\mu_1(g^{-1}x,g^{-1}y)$, \qquad 
$(g*\mu)(x,y,z)=g\mu_2(g^{-1}x,g^{-1}y,g^{-1}z)$,\end{center} 
and the orbits $\mathcal{O}(\mu)$ are the isomorphism classes.
We say that $\mu$ \emph{degenerates} to $\lambda$, and write $\mu\to\lambda$, if $\lambda$ belongs to
the Zariski closure $\overline{\mathcal{O}(\mu)}$.
An algebra $\mu$ is \emph{rigid} in a $\GL(\mathbb V)$-stable subvariety $\mathfrak X$ (for us,
$\mathfrak X=\mathfrak P_4$ or $\mathfrak X=\N_5$) if $\mathcal{O}(\mu)$ is open in $\mathfrak X$.
Since $\mathfrak X$ is a finite union of orbit closures and closures of families of orbits, each
irreducible component of $\mathfrak X$ is either the closure of the orbit of a rigid algebra or the
closure of a union of orbits of a family of algebras.
We use the following standard facts (see, for example, \cite{GRH,MS}):
\begin{enumerate}[(1)]
\item $\dim \mathcal{O}(\mu)=n^2-\dim\Der(\mu)$, and if $\mu\to\lambda$ with $\lambda\not\cong\mu$, then
$\dim \mathcal{O}(\lambda)<\dim \mathcal{O}(\mu)$;
\item if there are $E_1(t),\dots,E_n(t)\in\mathbb V$, $t\in\C^{*}$, forming a basis for every $t$,
such that the structure constants of $\mu$ in the basis $E_1(t),\dots,E_n(t)$ have limits as
$t\to0$ equal to the structure constants of $\lambda$ in the basis $e_1,\dots,e_n$, then
$\mu\to\lambda$;
\item \emph{(non-degenerations)} let $\mathcal B$ be the Borel subgroup of lower triangular
matrices and let $\mathcal R\subseteq\mathfrak P_n$ be a closed $\mathcal B$-stable subset; if
$\mu\to\lambda$ and $\mu$ is represented by a point of $\mathcal R$, then $\lambda$ is also
represented by a point of $\mathcal R$ \cite{GRH}; if $\mathcal R$ is $\GL(\mathbb V)$-stable
(that is, it is given by a property invariant under isomorphisms), then the condition on the basis
is not needed.
\end{enumerate}
To use (3) we write $\mathcal R$ by conditions on the structure constants.
Put $A_i=\langle e_i,\dots,e_n\rangle$ and write
$A_pA_q\subseteq A_r$ for the conditions $c_{ij}^k=0$ ($i\geq p$, $j\geq q$, $k<r$),
$[A_p,A_q,A_s]\subseteq A_r$ for the conditions $d_{ijl}^k=0$ ($i\geq p$, $j\geq q$, $l\geq s$,
$k<r$), and $A_pA_q=0$, $[A_p,A_q,A_s]=0$ if the corresponding constants vanish for all $k$.
A set given by such conditions is closed and $\mathcal B$-stable.
So, to prove $\mu\not\to\lambda$ it is enough to give such conditions which hold for $\mu$ in the
basis $e_1,\dots,e_n$ and fail for $\lambda$ in every basis.

\subsection{The variety of complex \texorpdfstring{$4$}{4}-dimensional \texorpdfstring{$3$}{3}-Poisson algebras}\label{sec:geo4}

In this subsection we describe the variety $\mathfrak P_4$ of all (not necessarily nilpotent)
complex $4$-dimensional $3$-Poisson algebras.
Besides the algebras of Theorem~A, it contains the commutative associative algebras (with zero
bracket) and the $3$-Lie algebras (with zero multiplication).
We need the following $3$-Lie algebras with zero commutative associative multiplication:
\begin{equation*}
\mathfrak s:\ [123]=e_4,\ [124]=-e_3,\ [134]=e_2,\ [234]=-e_1;\qquad
\mathfrak h^{\alpha}:\ [134]=e_1,\ [234]=\alpha e_2.
\end{equation*}
The algebra $\mathfrak s$ is the simple $4$-dimensional $3$-Lie algebra \cite{F85}.
For $\alpha\beta\neq0$ we have $\mathfrak h^{\alpha}\cong\mathfrak h^{\beta}$ if and only if
$\beta\in\{\alpha,\alpha^{-1}\}$ (see the proof of Lemma~\ref{lem:fil4}).

\begin{lemma}\label{lem:com4}
The variety ${\rm Com}_4\subseteq\mathfrak P_4$ of complex $4$-dimensional commutative associative
algebras (with zero bracket) is irreducible: ${\rm Com}_4=\overline{\mathcal{O}({\rm A}_{17})}$, where
${\rm A}_{17}=\C\oplus\C\oplus\C\oplus\C$ is the algebra of Theorem~\ref{4-dim asscom}, and
$\dim{\rm Com}_4=16$.
\end{lemma}

\begin{proof}
Since $\Der({\rm A}_{17})=0$, we have $\dim \mathcal{O}({\rm A}_{17})=16$.
Let ${\rm A}$ be a $4$-dimensional commutative associative algebra and let
${\rm A}^\sharp=\C1\oplus{\rm A}$ be the algebra obtained by adjoining a unit.
It is a unital commutative associative algebra of dimension $5$, and ${\rm A}$ is the kernel of the
character $\varepsilon:{\rm A}^\sharp\to\C$, $\varepsilon(1)=1$, $\varepsilon({\rm A})=0$.
The variety of unital commutative associative algebra structures on $\C^5$ (with a fixed unit) is
irreducible, and the orbit of $\C^5=\C\oplus\dots\oplus\C$ is dense in it \cite{Maz80}.
The pairs $({\rm B},\chi)$, where ${\rm B}$ is such a structure and $\chi$ is a character of ${\rm B}$,
form a closed subvariety $\mathcal X$ of the product of this variety with the dual space of $\C^5$,
and $\mathcal X$ is finite and flat over the variety of unital structures (it is the spectrum of the
universal algebra, which is a free module of rank $5$ over the coordinate ring).
A finite flat morphism is open and closed, so the preimage of the dense orbit of $\C^5$ is dense in
$\mathcal X$.
Hence $({\rm A}^\sharp,\varepsilon)$ is a limit of pairs $({\rm B}_t,\chi_t)$ with
${\rm B}_t\cong\C^5$.
The kernel of a character of $\C^5$ is isomorphic to ${\rm A}_{17}$, and the kernels
$\ker\chi_t$ tend to $\ker\varepsilon={\rm A}$; choosing bases of $\ker\chi_t$ depending
continuously on $t$, we obtain ${\rm A}\in\overline{\mathcal{O}({\rm A}_{17})}$.
\end{proof}

\begin{lemma}\label{lem:fil4}
The variety ${\rm Fil}_4\subseteq\mathfrak P_4$ of complex $4$-dimensional $3$-Lie algebras (with zero
multiplication) has two irreducible components,
\begin{equation*}
\overline{\mathcal{O}(\mathfrak s)}\quad\text{and}\quad\overline{\textstyle\bigcup_{\alpha\in\C}\mathcal{O}(\mathfrak h^{\alpha})},
\end{equation*}
of dimensions $10$ and $8$, respectively.
\end{lemma}

\begin{proof}
For $x,y,z\in\mathbb V$ denote by $x\times y\times z\in\mathbb V^*$ the linear form
$v\mapsto\det(x,y,z,v)$ (the determinant of the coordinates in the basis $e_1,\dots,e_4$).
Every skew-symmetric trilinear map on $\mathbb V$ is of the form
$[x,y,z]=M(x\times y\times z)$ for a unique linear map $M:\mathbb V^*\to\mathbb V$, that is, for a
unique $M\in\mathbb V\otimes\mathbb V$, and the action of $g\in\GL(\mathbb V)$ becomes
$M\mapsto\det(g)^{-1}gMg^{T}$.
Write $M=S+A$ with $S$ symmetric and $A$ skew-symmetric.
A direct computation shows that the Filippov identity holds if and only if
$A\wedge A=0$ and $S$ vanishes on the kernel of the skew-symmetric form dual to $A$; in other words,
either $A=0$, or $A=u\wedge v$ is decomposable and $S$ is a symmetric tensor in
$\langle u,v\rangle$.
Thus ${\rm Fil}_4=\mathcal S\cup\mathcal Y$, where $\mathcal S$ is the space of symmetric tensors
$M$ and $\mathcal Y$ is the set of tensors $M\in U\otimes U$ for some $2$-dimensional subspace
$U\subseteq\mathbb V$ (this agrees with the classification of the $(n+1)$-dimensional $n$-Lie
algebras \cite{F85}; see also \cite{BSZ11}).
The space $\mathcal S$ is irreducible of dimension $10$, and the non-degenerate symmetric tensors form
one orbit, namely $\mathcal{O}(\mathfrak s)$ (the algebra $\mathfrak s$ corresponds to $M=\Id$); so
$\mathcal S=\overline{\mathcal{O}(\mathfrak s)}$.
The set $\mathcal Y$ is the image of a vector bundle with fibre $U\otimes U$ of dimension $4$ over the
Grassmannian of $2$-planes in $\mathbb V$, which has dimension $4$; hence $\mathcal Y$ is closed and
irreducible of dimension at most $8$.
The algebra $\mathfrak h^{\alpha}$ corresponds to the tensor $M$ with $M(e_1^*)=-\alpha e_2$,
$M(e_2^*)=e_1$, $M(e_3^*)=M(e_4^*)=0$, which belongs to $U\otimes U$ for $U=\langle e_1,e_2\rangle$.
A generic element of $U\otimes U$ is equivalent to such a tensor, and if $M_U$ denotes the
$2\times2$ matrix of $M\in U\otimes U$, then the eigenvalues $-\alpha,-\alpha^{-1}$ of $M_U^{-1}M_U^{T}$ are
invariants of the action (up to the order).
Hence $\mathfrak h^{\alpha}\cong\mathfrak h^{\beta}$ ($\alpha\beta\neq0$) if and only if
$\beta\in\{\alpha,\alpha^{-1}\}$, $\dim \mathcal{O}(\mathfrak h^{\alpha})=7$ for $\alpha\neq1$, and
$\mathcal Y=\overline{\bigcup_{\alpha}\mathcal{O}(\mathfrak h^{\alpha})}$ has dimension $8$.
Finally, $\mathcal S\not\subseteq\mathcal Y$ by dimension, and $\mathcal Y\not\subseteq\mathcal S$,
since $\mathfrak h^{1}$ corresponds to a non-zero skew-symmetric tensor.
\end{proof}

\begin{theoremC}
The variety $\mathfrak P_4$ of complex $4$-dimensional $3$-Poisson algebras has dimension $16$ and it
has seven irreducible components:
\begin{gather*}
\overline{\mathcal{O}({\rm A}_{17})},\qquad
\overline{\mathcal{O}(\mathfrak s)},\qquad
\overline{\textstyle\bigcup_{\alpha\in\C}\mathcal{O}(\mathfrak h^{\alpha})},\qquad
\overline{\mathcal{O}(\PP^4_1)},\qquad
\overline{\mathcal{O}(\PP^4_4)},\qquad
\overline{\mathcal{O}(\PP^4_6)},\qquad
\overline{\mathcal{O}(\PP^4_7)},
\end{gather*}
of dimensions $16$, $10$, $8$, $10$, $10$, $10$ and $10$, respectively.
In particular, the rigid algebras of $\mathfrak P_4$ are ${\rm A}_{17}$, $\mathfrak s$, $\PP^4_1$,
$\PP^4_4$, $\PP^4_6$ and $\PP^4_7$.
\end{theoremC}

\begin{proof}
The orbit dimensions of the algebras which give the irreducible components (computed as $16-\dim\Der$)
and the dimensions of these components are given in Table below.
For the family $\mathfrak h^{\alpha}$ the parameter adds $1$: the orbits of $\mathfrak h^{\alpha}$ with
$\alpha\neq1$ have dimension $7$, only finitely many of them coincide, so the component has dimension
$7+1=8$; the point $\alpha=1$ is excluded in the table since $\dim\Der(\mathfrak h^{1})=11$ and
$\dim \mathcal{O}(\mathfrak h^{1})=5$.
\begin{longtable}{lccccccc}
\hline
& ${\rm A}_{17}$ & $\mathfrak s$ & $\mathfrak h^{\alpha}$, $\alpha\neq1$ & $\PP^4_1$ & $\PP^4_4$ & $\PP^4_6$ & $\PP^4_7$\\
\hline
\endfirsthead
\hline
& ${\rm A}_{17}$ & $\mathfrak s$ & $\mathfrak h^{\alpha}$, $\alpha\neq1$ & $\PP^4_1$ & $\PP^4_4$ & $\PP^4_6$ & $\PP^4_7$\\
\hline
\endhead
orbit & $16$ & $10$ & $7$ & $10$ & $10$ & $10$ & $10$\\
component & $16$ & $10$ & $7+1=8$ & $10$ & $10$ & $10$ & $10$\\
\hline
\end{longtable}

\emph{Every algebra belongs to one of the seven sets.}
Let ${\rm P}\in\mathfrak P_4$.
If the bracket of ${\rm P}$ is zero, then ${\rm P}\in\overline{\mathcal{O}({\rm A}_{17})}$ by
Lemma~\ref{lem:com4}; if the multiplication of ${\rm P}$ is zero, then ${\rm P}$ belongs to
$\overline{\mathcal{O}(\mathfrak s)}$ or to $\overline{\bigcup_\alpha \mathcal{O}(\mathfrak h^{\alpha})}$ by
Lemma~\ref{lem:fil4}.
Otherwise ${\rm P}$ is one of the algebras $\PP^4_1,\dots,\PP^4_7$ of Theorem~A, and the remaining
algebras $\PP^4_2$, $\PP^4_3$, $\PP^4_5$ are obtained by the degenerations of Table below.
\begin{longtable}{llllll}
 
\hline
$\PP^4_4\to\PP^4_2$ & $E_1=t^{-1}(e_1+e_2),$&$ E_2=2t^{-2}e_3,$&$ E_3=e_4,$&$ E_4=-te_2$\\
\hline
$\PP^4_6\to\PP^4_3$ & $E_1=t^{-1}e_3,$&$ E_2=t^{-2}e_4,$&$ E_3=e_1,$&$ E_4=t^{-1}e_2$\\
\hline
$\PP^4_6\to\PP^4_5$ & $E_1=t^{-1}e_1,$&$ E_2=t^{-1}e_2,$&$ E_3=t^{-2}e_4,$&$ E_4=e_3$\\
\hline
\end{longtable}
For example, in the first case
$E_1E_1=2t^{-2}e_3=E_2$, $E_1E_4=-e_3=-\frac{t^{2}}{2}E_2\to0$, and
$[E_1,E_3,E_4]=-[e_1+e_2,e_4,e_2]=e_4=E_3$, while all other products of the vectors $E_i$ vanish.

\emph{None of the seven sets is contained in another one.}
All seven sets are closed and irreducible, so it is enough to check that no algebra of the list
degenerates to another one (for the family, that the closure of the union of the orbits of the
family is not contained in another set and does not contain the other algebras).
\begin{enumerate}
\item The set of algebras with zero bracket is closed, so $\overline{\mathcal{O}({\rm A}_{17})}$ contains no
algebra with a non-zero bracket; the other six sets have dimension less than $16$, so they do not
contain ${\rm A}_{17}$.
\item The set of algebras with zero multiplication is closed, so $\overline{\mathcal{O}(\mathfrak s)}$ and
$\overline{\bigcup_\alpha \mathcal{O}(\mathfrak h^{\alpha})}$ do not contain $\PP^4_1,\PP^4_4,\PP^4_6,\PP^4_7$.
\item The orbits of $\mathfrak s,\PP^4_1,\PP^4_4,\PP^4_6,\PP^4_7$ have the same dimension $10$, so
there are no degenerations between these algebras, and the family set has dimension $8<10$.
\item The condition $\dim[{\rm P},{\rm P},{\rm P}]\leq1$ is closed and $\GL(\mathbb V)$-stable; it holds
for $\PP^4_1,\PP^4_4,\PP^4_6,\PP^4_7$ and fails for $\mathfrak h^{\alpha}$ with $\alpha\neq0$ (where
$[{\rm P},{\rm P},{\rm P}]=\langle e_1,e_2\rangle$).
Hence the family set is not contained in $\overline{\mathcal{O}(\PP^4_i)}$, $i=1,4,6,7$.
Similarly, $\mathfrak s$ is not in the family set by dimension, and the family set is not contained
in $\overline{\mathcal{O}(\mathfrak s)}$ by Lemma~\ref{lem:fil4}.
\end{enumerate}
Therefore the seven sets are the irreducible components of $\mathfrak P_4$, and
$\dim\mathfrak P_4=16$.
\end{proof}

\subsection{The variety of complex \texorpdfstring{$5$}{5}-dimensional nilpotent \texorpdfstring{$3$}{3}-Poisson algebras}\label{sec:geo5}

The orbit dimensions (computed as $25-\dim\Der$) of the algebras which give the irreducible
components of $\N_5$, and the dimensions of these components, are given in Table below.

\begin{longtable}{lccccc}
\hline
& $\AAA_2$ & $\PP_{07}$ & $\PP_{13}$ & $\PP_{19}$ & $\PP_{39}^{\alpha\notin\{0,1\}}$\\
\hline
\endfirsthead
\hline
& $\AAA_2$ & $\PP_{07}$ & $\PP_{13}$ & $\PP_{19}$ & $\PP_{39}^{\alpha\notin\{0,1\}}$\\
\hline
\endhead
orbit & $20$ & $18$ & $18$ & $18$ & $19$\\
component & $20$ & $18$ & $18$ & $18$ & $19+1=20$\\
\hline
\end{longtable}

For the family $\PP^{\alpha}_{39}$ the parameter adds $1$: for $\alpha\notin\{0,1\}$ we have
$\dim\Der(\PP^{\alpha}_{39})=6$, so $\dim \mathcal{O}(\PP^{\alpha}_{39})=19$, and since each orbit contains at
most six members of the family (Theorem~B), the union of the orbits has dimension $19+1=20$.
The two points $\alpha=0$ and $\alpha=1$ are excluded in the table because there the orbit is smaller:
$\PP^{0}_{39}\cong\PP^{1}_{39}$, $\dim\Der(\PP^{0}_{39})=7$ and $\dim \mathcal{O}(\PP^{0}_{39})=18$.
In terms of Case~4 of the proof of Theorem~B, the pencil of $\PP^{\alpha}_{39}$ is spanned by
$Q_1=x_1^2+\alpha x_2^2$ and the marked member $Q_2=x_1^2+x_2^2+x_3^2$, and
$\det(sQ_1+tQ_2)=t(s+t)(\alpha s+t)$ has a multiple root exactly for $\alpha\in\{0,1\}$; these
values do not affect the dimension of the component.
The variety of complex $5$-dimensional nilpotent commutative associative algebras is irreducible (see \cite{MS}) a
nd each algebra from this variety degenerates from

\begin{longtable}{ll}
\hline
$\mathcal A_{2}:$ & $e_{1} e_{1} = e_{2},\ e_{1} e_{2} = e_{3},\ e_{1} e_{3} = e_{4},\ e_{1} e_{4} = e_{5},\ e_{2} e_{2} = e_{4},\ e_{2} e_{3} = e_{5}$ \\
\hline
\end{longtable}

\begin{theoremD}
The variety $\N_5$ of complex $5$-dimensional nilpotent $3$-Poisson algebras has dimension $20$ and
it has five irreducible components:
\begin{equation*}
\overline{\mathcal{O}(\AAA_2)},\qquad
\overline{\textstyle\bigcup_{\alpha\in\C}\mathcal{O}(\PP_{39}^{\alpha})},\qquad
\overline{\mathcal{O}(\PP_{07})},\qquad
\overline{\mathcal{O}(\PP_{13})},\qquad
\overline{\mathcal{O}(\PP_{19})},
\end{equation*}
of dimensions $20$, $20$, $18$, $18$ and $18$, respectively.
In particular, the rigid algebras of $\N_5$ are $\AAA_2$, $\PP_{07}$, $\PP_{13}$ and $\PP_{19}$.
\end{theoremD}

\begin{proof}
For $k=1,2$ let $\mathcal Y_k\subseteq\N_5$ be the set of algebras ${\rm P}$ having a
$k$-dimensional subspace $U$ with ${\rm P}^2\subseteq U\subseteq\Ann({\rm P})$.
For such a pair $({\rm P},U)$ all products of three elements vanish, so every pair of a symmetric
bilinear map and a skew-symmetric trilinear map ${\rm P}/U\to U$ defines a nilpotent $3$-Poisson
algebra.
Hence $\mathcal Y_k$ is the image of a vector bundle over the Grassmannian $G_k(\mathbb V)$ with
fibre of dimension $k\big(\binom{6-k}{2}+\binom{5-k}{3}\big)$; it is closed (the Grassmannian is
projective) and irreducible.
Since $U={\rm P}^2$ for a generic point, $\dim\mathcal Y_1=4+14=18$ and $\dim\mathcal Y_2=6+14=20$.
Now $\PP_{07}\in\mathcal Y_1$ and $\dim \mathcal{O}(\PP_{07})=18$, so $\mathcal Y_1=\overline{\mathcal{O}(\PP_{07})}$.
The algebras $\PP_{39}^{\alpha}$ belong to $\mathcal Y_2$; their orbits have dimension $19$ for
$\alpha\neq0,1$, and two of them are isomorphic only for finitely many values of the parameter, so
$\bigcup_{\alpha}\mathcal{O}(\PP^{\alpha}_{39})$ has dimension $20$ and
$\mathcal Y_2=\overline{\bigcup_{\alpha}\mathcal{O}(\PP_{39}^{\alpha})}$.

If $\dim({\rm P}\cdot{\rm P})\geq4$, then ${\rm P}^2={\rm P}\cdot{\rm P}$ has codimension $1$ and
${\rm P}$ is generated by one element $x$; by the Leibniz rule and the skew-symmetry,
$[x^a,x^b,x^c]=abc\,x^{a+b+c-3}\cdot[x,x,x]=0$, so ${\rm P}\cong\AAA_2$.
Hence $\mathcal{O}(\AAA_2)$ is open, $\AAA_2$ is rigid, and $\dim \mathcal{O}(\AAA_2)=20$.

Next we show that every algebra of $\N_5$ belongs to the union of the five sets of the theorem.
The variety of $5$-dimensional nilpotent commutative associative algebras is irreducible and it is
the orbit closure of the null-filiform algebra $\AAA_2$ \cite{Maz80}.
So all $\AAA_i$ lie in $\overline{\mathcal{O}(\AAA_2)}$.
By the classification, the algebras $\PP_{21},\dots,\PP_{39}^{\alpha}$, $\PP_{01}$, $\PP_{02}$ belong to
$\mathcal Y_2$, and the algebras $\PP_{03},\dots,\PP_{10}$ and $\mathfrak n_1\oplus\C$ belong to
$\mathcal Y_1$.
For the remaining algebras we use the degenerations

{\small
\begin{longtable}{l|lllll}
\hline
$\mathcal P_{19}\to \mathcal P_{20}$ & 
\multicolumn{5}{|l}{$E_{1}=t^{-2}(e_{1} -  \frac{1}{2} e_{2} + e_{3}),\ 
E_{2}=t(e_{1} + \frac{1}{2} e_{2}),\ 
E_{3}=t^{4} e_{3},\ 
E_{4}=t^{3} e_{4},\ 
E_{5}=t^{2} e_{5}$} \\
\hline
$\mathcal P_{19}\to \mathcal P_{17}$ & $E_{1}=t\,e_{1},$ & $ E_{2}=t^{-2}(\frac{1}{2} e_{1} + e_{2}),$ & $ E_{3}=\frac{1}{t^{2}} e_{3},$ & $ E_{4}=\frac{1}{t^{3}} e_{4},$ & $ E_{5}=\frac{1}{t^{4}} e_{5}$ \\
\hline
$\mathcal P_{13}\to \mathcal P_{12}$ & $E_{1}=\frac{1}{t^{2}} e_{1},$ & $ E_{2}=\frac{1}{t^{4}} e_{2},$ & $ E_{3}=\frac{1}{t^{3}} e_{3},$ & $ E_{4}=\frac{1}{t} e_{4},$ & $ E_{5}=\frac{1}{t^{6}} e_{5}$ \\
\hline
$\mathcal P_{12}\to \mathcal P_{11}$ & $E_{1}=e_{1},$ & 
$ E_{2}=e_{2},$ & 
$ E_{3}=t\,e_{3},$ & 
$ E_{4}=\frac{1}{t} e_{4},$ & 
$ E_{5}=e_{5}$ \\
\hline
$\mathcal P_{14}\to \mathfrak n_2$ & $E_{1}=t\,e_{1},$ & 
$ E_{2}=\frac{1}{t} e_{2},$ & 
$ E_{3}=e_{3},$ & 
$ E_{4}=e_{4},$ & 
$ E_{5}=e_{5}$ \\
\hline
$\mathcal P_{20}\to \mathcal P_{18}$ & $E_{1}=e_{1},$ & 
$ E_{2}=-e_{2},$ & 
$ E_{3}=\frac{1}{t} e_{3},$ & 
$ E_{4}=- \frac{1}{t} e_{4},$ & 
$ E_{5}=\frac{1}{t} e_{5}$ \\
\hline
$\mathcal P_{19}\to \mathcal P_{16}$ & $E_{1}=- \frac{1}{t} e_{1},$ & $ E_{2}=-e_{2},$ & 
$ E_{3}=t\,e_{3},$ & 
$ E_{4}=e_{4},$ & 
$ E_{5}=\frac{1}{t} e_{5}$ \\
\hline
$\mathcal P_{17}\to \mathcal P_{15}$ & $E_{1}=- \frac{1}{t} e_{1},$ & 
$ E_{2}=t(-e_{1} - e_{2}),$ & 
$ E_{3}=e_{3},$ & 
$ E_{4}=e_{4},$ & 
$ E_{5}=e_{5}$ \\
\hline
$\mathcal P_{20}\to \mathcal P_{14}$ & $E_{1}=- \frac{1}{t} e_{2},$ & 
$ E_{2}=e_{1},$ & 
$ E_{3}=e_{3},$ & 
$ E_{4}=\frac{1}{t} e_{4},$ & 
$ E_{5}=\frac{1}{t^{2}} e_{5}$ \\
\hline
\end{longtable}
}

Finally, none of the five sets is contained in another one.
Since all five sets are irreducible and closed, it is enough to prove the twenty
non-degenerations

{\small
\begin{longtable}{rcl|l}
\hline
\endhead
$\AAA_2$&$\not \to $&$
\big\{\PP_{07},\PP_{13},\PP_{19},\PP_{39}^{\alpha}\big\}$ & $[A_1,A_1,A_1]=0$\\
\hline
$\PP_{39}^{\alpha}$&$\not\to$&$ 
\big\{\AAA_2,\PP_{07},\PP_{13},\PP_{19} \big\}$ & 
$A_1A_4=0$, $[A_1,A_1,A_4]=0$\\
\hline
$\PP_{07}$&$\not\to$&$ 
\big\{\AAA_2,\PP_{13},\PP_{19},\PP_{39}^{\alpha}\big\}$ & $A_1A_1\subseteq A_5$, $[A_1,A_1,A_1]\subseteq A_5$\\
\hline
$\PP_{13}$&$\not\to$&$\big\{\AAA_2,\PP_{07},\PP_{19},\PP_{39}^{\alpha}\big\}$ & $A_1A_1\subseteq A_4$, $A_1A_2\subseteq A_5$, $A_2A_4=0$, $A_1A_5=0$,\\
&&&$[A_1,A_1,A_1]\subseteq A_5$, $[A_1,A_1,A_4]=0$\\
\hline
$\PP_{19}$&$\not\to$&$\big\{\AAA_2,\PP_{07},\PP_{13},\PP_{39}^{\alpha}\big\}$ & $A_1A_1\subseteq A_5$, $A_1A_4=0$, $[A_1,A_1,A_1]\subseteq A_4$,\\
&&&$[A_1,A_1,A_4]\subseteq A_5$, $[A_1,A_3,A_4]=0$, $[A_1,A_1,A_5]=0$\\
\hline
\end{longtable}
}

(for the family we use the closure of the union of
all orbits $\mathcal{O}(\PP_{39}^{\alpha})$; the set $\mathcal R$ contains all $\PP_{39}^{\alpha}$, so it contains
this closure).
In each line the algebra on the left satisfies the conditions in the basis $e_1,\dots,e_5$
(for $\PP_{13}$ we use the basis $e_1,e_3,e_4,e_2,e_5$, in which
$e_1e_1=e_4$, $e_1e_4=e_5$, $e_2e_2=e_3e_3=e_5$, $[123]=e_5$).
The list below explains in simple terms why the algebras on the right do not satisfy the
conditions in any basis $f_1,\dots,f_5$.
\begin{enumerate}
\item $\AAA_2$: the conditions say that the $3$-Lie multiplication is zero, but in
$\PP_{07},\PP_{13},\PP_{19},\PP_{39}^{\alpha}$ it is not zero.
\item $\PP_{39}^{\alpha}$: the conditions say that $f_4,f_5\in\Ann$, but
$\Ann=\langle e_5\rangle$ in $\AAA_2,\PP_{07},\PP_{13},\PP_{19}$.
\item $\PP_{07}$: the conditions say that ${\rm P}^2\subseteq\langle f_5\rangle$, but
$\dim{\rm P}^2\geq2$ in $\AAA_2,\PP_{13},\PP_{19},\PP_{39}^{\alpha}$.
\item $\PP_{13}$: the conditions give ${\rm P}\cdot{\rm P}\subseteq\langle f_4,f_5\rangle$ and
$[{\rm P},{\rm P},{\rm P}]\subseteq\langle f_5\rangle$; this excludes $\AAA_2$
($\dim\AAA_2\cdot\AAA_2=4$) and $\PP_{19}$ ($\dim[{\rm P},{\rm P},{\rm P}]=2$).
For $\PP_{39}^{\alpha}$ the bracket gives $f_5\in\langle e_4\rangle$; then $f_j\cdot{\rm P}\subseteq
\langle f_5\rangle=\langle e_4\rangle$ for $j\geq2$, so the $e_5$-coordinate
$x_1y_1+x_2y_2+x_3y_3$ of $x\cdot y$ vanishes for $x=f_j$, $j\geq 2$, and all $y$; hence
$f_2,f_3,f_4,f_5\in\langle e_4,e_5\rangle$, which is impossible.
For $\PP_{07}$ the conditions give $f_5\in\Ann=\langle e_5\rangle$, and
$[f_4,{\rm P},{\rm P}]=0$, $f_4f_4=0$ give $f_4\in\langle e_4,e_5\rangle$ and then
$f_4\in\langle e_5\rangle$, which is impossible.
\item $\PP_{19}$: the conditions give ${\rm P}\cdot{\rm P}\subseteq\langle f_5\rangle$ and
$f_4\cdot{\rm P}=f_5\cdot{\rm P}=0$; this excludes $\AAA_2,\PP_{13},\PP_{39}^{\alpha}$
($\dim{\rm P}\cdot{\rm P}\geq2$) and $\PP_{07}$ (where $\{x: x\cdot{\rm P}=0\}=\langle e_5\rangle$).
\end{enumerate}
This proves the theorem.
\end{proof}


\begin{thebibliography}{99}

\bibitem{ABCF23}
Abdelwahab H., Barreiro E., Calder\'on A.J., Fern\'andez Ouaridi A.,
The algebraic classification and degenerations of nilpotent Poisson algebras,
Journal of Algebra, 615 (2023), 243--277.
\doi{10.1016/j.jalgebra.2022.10.018}

\bibitem{AFM25}
Abdelwahab H., Fern\'andez Ouaridi A., Mart\'in Gonz\'alez C.,
Degenerations of Poisson algebras,
Journal of Algebra and its Applications, 24 (2025), 3, 2550087.
\doi{10.1142/S0219498825500872}

\bibitem{aae23}
Abdurasulov K., Adashev J., Eshmeteva S.,
Transposed Poisson structures on solvable Lie algebras with filiform nilradical,
Communications in Mathematics, 32 (2024), 3, 441--483.
\doi{10.46298/cm.12844}

\bibitem{kz}
Abdurasulov K., Khudoyberdiyev A., Toshtemirova F.,
The geometric classification of nilpotent Lie--Yamaguti, Bol and compatible Lie algebras,
Communications in Mathematics, 33 (2025), 3, 10.
\doi{10.46298/cm.15810}

\bibitem{BL08}
Bagger J., Lambert N.,
Gauge symmetry and supersymmetry of multiple M2-branes,
Physical Review D, 77 (2008), 6, 065008.
\doi{10.1103/PhysRevD.77.065008},
\arxiv{0711.0955}

\bibitem{BSZ11}
Bai R., Song G., Zhang Y.,
On classification of $n$-Lie algebras,
Frontiers of Mathematics in China, 6 (2011), 4, 581--606.
\doi{10.1007/s11464-011-0107-z}

\bibitem{BH05}
Basu A., Harvey J.,
The M2--M5 brane system and a generalized Nahm's equation,
Nuclear Physics B, 713 (2005), 1--3, 136--150.
\doi{10.1016/j.nuclphysb.2005.02.007},
\arxiv{hep-th/0412310}

\bibitem{ben}
Ben Hassine A., Chtioui T., Elhamdadi M., Mabrouk S.,
Cohomology and deformations of left-symmetric Rinehart algebras,
Communications in Mathematics, 32 (2024), 2, 127--152.
\doi{10.46298/cm.12656}

\bibitem{BCZ25}
Bi Y., Chen Z., Zhang T.,
On (co-)morphisms of $n$-Lie--Rinehart algebras with applications to Nambu--Poisson manifolds,
Journal of Geometry and Physics, 214 (2025), 105536.
\doi{10.1016/j.geomphys.2025.105536}

\bibitem{BG13}
Burde D., de Graaf W.,
Classification of Novikov algebras,
Applicable Algebra in Engineering, Communication and Computing, 24 (2013), 1, 1--15.
\doi{10.1007/s00200-012-0180-x}

\bibitem{BC99}
Burde D., Steinhoff C.,
Classification of orbit closures of $4$-dimensional complex Lie algebras,
Journal of Algebra, 214 (1999), 2, 729--739.
\doi{10.1006/jabr.1998.7714}

\bibitem{CK16}
Cantarini N., Kac V.,
Classification of linearly compact simple Nambu--Poisson algebras,
Journal of Mathematical Physics, 57 (2016), 5, 051701.
\doi{10.1063/1.4948409}

\bibitem{D18}
Das A.,
Reduction of Nambu--Poisson manifolds by regular distributions,
Mathematical Physics, Analysis and Geometry, 21 (2018), 1, 5.
\doi{10.1007/s11040-018-9264-6}

\bibitem{DET26}
Daukeyeva N., Eraliyeva M., Toshtemirova F.,
Transposed $\delta$-Poisson algebra structures on null-filiform associative algebras,
Communications in Mathematics, 34 (2026), 1, 1.
\doi{10.46298/cm.15974}

\bibitem{DPP96}
de Azc\'arraga J., Perelomov A., P\'erez Bueno J.,
New generalized Poisson structures,
Journal of Physics A: Mathematical and General, 29 (1996), 7, L151--L157.
\doi{10.1088/0305-4470/29/7/001},
\arxiv{q-alg/9601007}

\bibitem{DSS10}
DeBellis J., S\"amann C., Szabo R.,
Quantized Nambu--Poisson manifolds and $n$-Lie algebras,
Journal of Mathematical Physics, 51 (2010), 12, 122303.
\doi{10.1063/1.3503773},
\arxiv{1001.3275}

\bibitem{DSS11}
DeBellis J., S\"amann C., Szabo R.,
Quantized Nambu--Poisson manifolds in a $3$-Lie algebra reduced model,
Journal of High Energy Physics, 2011 (2011), 4, 75.
\doi{10.1007/JHEP04(2011)075},
\arxiv{1012.2236}

\bibitem{DS17}
de Le\'on M., Sard\'on C.,
Geometric Hamilton--Jacobi theory on Nambu--Poisson manifolds,
Journal of Mathematical Physics, 58 (2017), 3, 033508.
\doi{10.1063/1.4978853}

\bibitem{DFST97}
Dito G., Flato M., Sternheimer D., Takhtajan L.,
Deformation quantization and Nambu mechanics,
Communications in Mathematical Physics, 183 (1997), 1, 1--22.
\doi{10.1007/s002200050017},
\arxiv{hep-th/9602016}

\bibitem{MF}
Fairon M.,
Modified double brackets and a conjecture of S.~Arthamonov,
Communications in Mathematics, 33 (2025), 3, 5.
\doi{10.46298/cm.13786}

\bibitem{FKS}
Fehlberg J\'unior R., Kaygorodov I., Saydaliyev A.,
The geometric classification of symmetric Leibniz algebras,
Communications in Mathematics, 33 (2025), 1, 10.
\doi{10.46298/cm.15949}

\bibitem{FKS25}
Fehlberg J\'unior R., Kaygorodov I., Saydaliyev A.,
The complete classification of irreducible components of varieties of Jordan superalgebras,
Communications in Mathematics, 33 (2025), 3, 15.
\doi{10.46298/cm.16917}

\bibitem{F85}
Filippov V.,
$n$-Lie algebras,
Siberian Mathematical Journal, 26 (1985), 6, 879--891.
\doi{10.1007/BF00969110}

\bibitem{GG08}
Gauntlett J., Gutowski J.,
Constraining maximally supersymmetric membrane actions,
Journal of High Energy Physics, 2008 (2008), 6, 053.
\doi{10.1088/1126-6708/2008/06/053},
\arxiv{0804.3078}

\bibitem{GM00}
Grabowski J., Marmo G.,
On Filippov algebroids and multiplicative Nambu--Poisson structures,
Differential Geometry and its Applications, 12 (2000), 1, 35--50.
\doi{10.1016/S0926-2245(99)00042-X}

\bibitem{GRH}
Grunewald F., O'Halloran J.,
Varieties of nilpotent Lie algebras of dimension less than six,
Journal of Algebra, 112 (1988), 2, 315--325.
\doi{10.1016/0021-8693(88)90093-2}

\bibitem{GRH3}
Grunewald F., O'Halloran J.,
Deformations of Lie algebras,
Journal of Algebra, 162 (1993), 1, 210--224.
\doi{10.1006/jabr.1993.1250}

\bibitem{G09}
Gustavsson A.,
Algebraic structures on parallel M2-branes,
Nuclear Physics B, 811 (2009), 1--2, 66--76.
\doi{10.1016/j.nuclphysb.2008.11.014},
\arxiv{0709.1260}

\bibitem{HS25}
Harrathi F., Sendi S.,
$\mathcal O$-operators, $r$-matrices and splitting of operations for non-commutative ternary
Nambu--Poisson algebras,
Journal of Geometry and Physics, 217 (2025), 105608.
\doi{10.1016/j.geomphys.2025.105608}

\bibitem{HC23}
Herbig H.-C., Chaparro Casta\~neda A.M.,
Higher form brackets for even Nambu--Poisson algebras,
Letters in Mathematical Physics, 113 (2023), 5, 97.
\doi{10.1007/s11005-023-01718-3}

\bibitem{IDMM97}
Ib\'a\~nez R., de Le\'on M., Marrero J.C., Mart\'in de Diego D.,
Dynamics of generalized Poisson and Nambu--Poisson brackets,
Journal of Mathematical Physics, 38 (1997), 5, 2332--2344.
\doi{10.1063/1.531960}

\bibitem{IW53}
In\"on\"u E., Wigner E.,
On the contraction of groups and their representations,
Proceedings of the National Academy of Sciences of the USA, 39 (1953), 6, 510--524.
\doi{10.1073/pnas.39.6.510}

\bibitem{k23}
Kaygorodov I.,
Non-associative algebraic structures: classification and structure,
Communications in Mathematics, 32 (2024), 3, 1--62.
\doi{10.46298/cm.11419}

\bibitem{MS}
Kaygorodov I., Khrypchenko M., P\'aez-Guill\'an P.,
The geometric classification of non-associative algebras: a survey,
Communications in Mathematics, 32 (2024), 2, 185--284.
\doi{10.46298/cm.14458}

\bibitem{KV20}
Kaygorodov I., Volkov Yu.,
Degenerations of Filippov algebras,
Journal of Mathematical Physics, 61 (2020), 2, 021701.
\doi{10.1063/1.5119393}

\bibitem{L23}
Lacroix S.,
On a class of conformal $\mathcal E$-models and their chiral Poisson algebras,
Journal of High Energy Physics, 2023 (2023), 6, 45.
\doi{10.1007/JHEP06(2023)045}

\bibitem{l24}
Lopes S.,
Noncommutative algebra and representation theory: symmetry, structure \& invariants,
Communications in Mathematics, 32 (2024), 3, 63--117.
\doi{10.46298/cm.11678}

\bibitem{LS26}
Lubkov R., Shiringovskiy I.,
Transposed Poisson structure on the Witt-type algebra $\mathcal W(a,-1)$: derivations, automorphisms,
and Rota--Baxter operators,
Communications in Mathematics, 34 (2026), 1.
\doi{10.46298/cm.18245}

\bibitem{MA15}
Makhlouf A., Amri A.,
Non-commutative ternary Nambu--Poisson algebras and ternary Hom-Nambu--Poisson algebras,
Journal of Generalized Lie Theory and Applications, 9 (2015), 1, 1000221.
\doi{10.4172/1736-4337.1000221}

\bibitem{Maz80}
Mazzola G.,
Generic finite schemes and Hochschild cocycles,
Commentarii Mathematici Helvetici, 55 (1980), 1, 267--293.
\doi{10.1007/BF02566686}

\bibitem{N98}
Nakanishi N.,
On Nambu--Poisson manifolds,
Reviews in Mathematical Physics, 10 (1998), 4, 499--510.
\doi{10.1142/S0129055X98000161}

\bibitem{N73}
Nambu Y.,
Generalized Hamiltonian dynamics,
Physical Review D, 7 (1973), 8, 2405--2412.
\doi{10.1103/PhysRevD.7.2405}

\bibitem{OS24}
Olshanski G., Safonkin N.,
Double Poisson brackets and involutive representation spaces,
Letters in Mathematical Physics, 114 (2024), 1, 33.
\doi{10.1007/s11005-024-01782-3}

\bibitem{P08}
Papadopoulos G.,
M2-branes, $3$-Lie algebras and Pl\"ucker relations,
Journal of High Energy Physics, 2008 (2008), 5, 054.
\doi{10.1088/1126-6708/2008/05/054},
\arxiv{0804.2662}

\bibitem{Poo08}
Poonen B.,
Isomorphism types of commutative algebras of finite rank over an algebraically closed field,
Contemporary Mathematics, 463 (2008), 111--120.
\doi{10.1090/conm/463/09050}

\bibitem{kms}
Sartayev B.,
Some generalizations of the variety of transposed Poisson algebras,
Communications in Mathematics, 32 (2024), 2, 55--62.
\doi{10.46298/cm.11346}

\bibitem{SS78}
Skjelbred T., Sund T.,
Sur la classification des alg\`ebres de Lie nilpotentes,
Comptes Rendus de l'Acad\'emie des Sciences Paris, S\'erie A-B, 286 (1978), 5, A241--A242.

\bibitem{SJ18}
Song L., Jiang J.,
Generalized derivation extensions of $3$-Lie algebras and corresponding Nambu--Poisson structures,
Journal of Geometry and Physics, 124 (2018), 74--85.
\doi{10.1016/j.geomphys.2017.10.011}

\bibitem{T94}
Takhtajan L.,
On foundation of the generalized Nambu mechanics,
Communications in Mathematical Physics, 160 (1994), 2, 295--315.
\doi{10.1007/BF02103278},
\arxiv{hep-th/9301111}

\bibitem{Th91}
Thompson R.,
Pencils of complex and real symmetric and skew matrices,
Linear Algebra and its Applications, 147 (1991), 323--371.
\doi{10.1016/0024-3795(91)90238-R}

\end{thebibliography}
\end{document}